\documentclass[11pt,amsfonts]{article}
\usepackage[margin=1in]{geometry}  
\usepackage{graphicx}              

\usepackage{amsmath, amssymb, amsthm, epsfig}               
\usepackage{enumerate}
\usepackage{amsfonts}              
\usepackage{amsthm}                
\usepackage{amsthm}
\usepackage{enumerate}
\theoremstyle{plain}
\usepackage{color}
\newtheorem{corollary}{Corollary}[section]
\newtheorem{definition}[corollary]{Definition}
\newtheorem{definitions}[corollary]{Definitions}

\newtheorem{lemma}[corollary]{Lemma}
\newtheorem{prop}[corollary]{Proposition}
\newtheorem{rem}[corollary]{Remark}
\newtheorem{thm}[corollary]{Theorem}

\newfont{\sBlackboard}{msbm10 scaled 900}

\newcommand{\mylabel}[1]{\label{#1}
    \ifx\undefined\stillediting
    \else \fbox{$#1$}\fi }
\newcommand{\BE}{\begin{equation}}

\newcommand{\EEQ}{\end{equation}}
\newcommand{\rfb}[1]{\mbox{\rm
        (\ref{#1})}\ifx\undefined\stillediting\else:\fbox{$#1$}\fi}

\newfont{\Blackboard}{msbm10 scaled 1200}

\newfont{\roma}{cmr10 scaled 1200}

\newcommand{\bb}{\begin{equation}}
\newcommand{\bbb}{\end{equation}}

\newcommand{\mm}    {{\hbox{\hskip 0.5pt}}}

\newcommand{\bluff} {{\hbox{\raise 15pt \hbox{\mm}}}}
scaled\magstep2

\makeatletter
\def\section{\@startsection {section}{1}{\z@}{-3.5ex plus -1ex minus
        -.2ex}{2.3ex plus .2ex}{\large\bf}}
\makeatother

\begin{document}

\title{\sc Orlicz-Sobolev versus H\"older local minimizers for nonlinear Robin problems}
\author{Anouar Bahrouni, Hlel Missaoui, Hichem Ounaies and Vicen\c{t}iu
D.\,R\u{a}dulescu} \maketitle
\begin{abstract}
We establish regularity results for weak solutions of Robin problems driven by the well-known Orlicz $g$-Laplacian operator given by
\begin{equation}\label{P1}
\left\lbrace
\begin{array}{ll}
 -\Delta _g u=f(x,u),& x\in\Omega\\
\displaystyle a(\vert \nabla u\vert)\frac{\partial u}{d\nu}+b(x)\vert u\vert^{p-2}u=0,& x\in \partial \Omega,
\end{array}
 \right.\tag{P}
\end{equation}
where $\Delta _g u:=\text{div}(a(\vert\nabla u\vert)\nabla u)$, $\Omega\subset \mathbb{R}^N,\ N\geq 3$, is a bounded domain with $C^2$-boundary $\partial\Omega$, $\frac{\partial u}{d\nu}=\nabla u \cdot\nu$,
 $\nu$ is the unit exterior vector on $\partial\Omega$, $p>0$, $b \in C^{1,\theta}(\partial\Omega)$ with $\theta\in(0,1)$ and $\inf_{x\in \partial\Omega} b(x) > 0$. Specifically, using a suitable variation of the Moser iteration technique, we prove that every weak solution of the problem $(\ref{P})$ is bounded. Moreover, we combine this result with the Lieberman regularity theorem, to show that every $C^1(\overline{\Omega})$-local minimizer is also a $W^{1,G}(\Omega)$-local minimizer for the corresponding energy functional of problem $(\ref{P})$.
 
 
 \smallskip\noindent{\bf Keywor{\rm d}s:} Orlicz-Sobolev space,
 	Orlicz $g$-Laplacian, Robin boundary values, Moser iteration.
\end{abstract}

\section{Introduction}
In this paper, we study the boundedness regularity for a weak solution and the relationship between the H\"older local minimizer and the Orlicz-Sobolev local minimizer for the corresponding energy functional of the following Robin problem:
\begin{equation}\label{P}
\left\lbrace
\begin{array}{ll}
 -\Delta _g u=f(x,u),& \text{on}\  \Omega\\
\displaystyle a(\vert\nabla u\vert)\frac{\partial u }{d\nu}+b(x)\vert u\vert^{p-2}u=0,& \text{on}\  \partial \Omega,
\end{array}
 \right.\tag{P}
\end{equation}
where $\Omega$ is a bounded open subset of $\mathbb{R}^N$ $(N\geq3)$ with $C^2$-boundary $\partial\Omega$, $\Delta _g u:=\text{div}(a(\vert\nabla u\vert)\nabla u)$ is the Orlicz $g$-Laplacian  operator, $\frac{\partial u}{d\nu}=\nabla u .\nu$, $\nu$ is the unit exterior vector on $\partial\Omega$, $p>0$,  $b \in C^{1,\theta}(\partial\Omega)$ with $\theta\in(0,1)$ and $\inf\limits_{x\in \partial\Omega} b(x) > 0$ and the function $a(\vert t\vert)t$ is an increasing homeomorphism from $\mathbb{R}$ onto $\mathbb{R}$. In the right side of problem $(\ref{P})$ there is a Carath\'eodory function $f:\Omega\times\mathbb{R}\longrightarrow \mathbb{R}$, that is $x\longmapsto f(x,s)$ is measurable for all $s\in\mathbb{R}$ and $s\longmapsto f(x,s)$ continuous for a.e. $x\in\Omega$.\\

Due to the nature of the non-homogeneous differential operator $g$-Laplacian, we shall work in the framework of Orlicz and Orlicz-Sobolev spaces. The study of variational problems in the classical Sobolev and Orlicz-Sobolev spaces is an interesting topic of research due to its significant role in many fiel{\rm d}s of mathematics, such as approximation theory, partial differential equations, calculus of variations, non-linear potential theory, the theory of quasi-conformal mappings, non-Newtonian flui{\rm d}s, image processing, differential geometry, geometric function theory, and probability theory (see \cite{16,17,18,19,6}).\\

It is worthwhile to mention that the Orlicz-Sobolev space is a generalization of the classical Sobolev
space. Hence, several properties of the Sobolev spaces have been extended to the Orlicz-Sobolev spaces.
To the best of our knowledge, there is a lack of some regularity results concerning the problem $(\ref{P})$. Precisely, the boundedness of a weak solution and the relationship between the Orlicz-Sobolev and H\"older local minimizers for the corresponding energy functional of $(\ref{P})$. Those results are crucial in some metho{\rm d}s of the existence and multiplicity of solutions for the problem $(\ref{P})$.\\

The question of the boundedness, regularity, and the relationship between the Sobolev and H\"older local minimizers for certain $C^1$-functionals have been treated by many authors \cite{ylp,29,fan2,6,3,piral,30,lmt,31,shu,1,2,mgpw,4,pwa,patrik1,patrik2} and references therein.
In \cite{1}, G. M. Lieberman treated the regularity result up to the boundary for the weak solutions of the following problem
\begin{equation}\label{E}
\begin{array}{ll}
 -\Delta _p u=f(x,u),& x\in\Omega\\
\end{array}
 \tag{E}
\end{equation}
where $\Omega$ is a bounded domain in $\mathbb{R}^N$ with $C^{1,\alpha}$-boundary. Precisely, under some assumptions on the structure of the $p$-Laplacian operator and on the non-linear term  $f$, he proved that every bounded (i.e. $u\in L^{\infty}(\Omega)$) weak solution of the problem $(\ref{E})$ (with Dirichlet or Neumann boundary conditions) belongs to $C^{1,\beta}(\overline{\Omega})$. In \cite{2}, G. M. Lieberman, extended the results obtained in \cite{1} to the Orlicz $g$-Laplacian operator. In \cite{fan2}, X. L. Fan, established the same results gave in \cite{1} for the variable exponent Sobolev spaces ($p$ being variable). Note that all the results cited in \cite{fan2,1,2} require that the weak solution belongs to $L^{\infty}(\Omega)$. The boundedness result for weak solutions in the Dirichlet case can be deduced from Theorem $7.1$ of Ladyzhenskaya-Uraltseva \cite{lad} (problems with standard growth conditions) and Theorem $4.1$ of Fan-Zhao \cite{fan3} (problems with non-standard growth conditions). For the Neumann case, the boundedness result is deduced from Proposition 3.1 of Gasi\`nski-Papageorgiou \cite{3} (problems with sub-critical growth conditions).\\

 To the best of our knowledge, there is only one paper (see \cite{6}) devoted to the boundedness result of weak solutions to problems driven by the Orlicz $g$-Laplacian operator. Precisely, in \cite{6}, F. Fang and Z. Tan, with sub-critical growth conditions, proved that every weak solution of problems with Dirichlet boundary conditions belongs to $L^{\infty}(\Omega)$.  The approaches used by Fang and Tan in \cite{6} for the boundedness result don't work in our case (Robin boundary condition) since they require that $u_{\mid_{\partial\Omega}}$ is bounded ($u$ being the weak solution). To overcome this difficulty, we apply a suitable variation of the Moser iteration technique.\\

The question of the relationship between the Sobolev and H\"older local minimizers for certain functionals has taken the attention of many authors \cite{ylp,29,6,3,piral,30,34,31,shu,iann,mom,patrik1,patrik2} and references therein. In \cite{29}, Brezis and Nirenberg have proved a famous theorem which asserts that the local minimizers in the space $C^1$ are also local minimizers in the space $H^1$ for certain variational functionals. A result of this type was later extended to the space $W^{1,p}_0(\Omega)$ ( Dirichlet boundary condition), with $1<p<\infty$, by Garcia Azorera-Manfredi-Peral Alonso \cite{30} (see also Guo-Zhang \cite{31}, where $2\leq p$). The $W^{1,p}_n(\Omega)$-version (Neumann boundary condition) of the result can be found in Motreanu-Motreanu-Papageorgiou \cite{mom}. Moreover,
this theorem has been extended to the  $p(x)$-Laplacian equations (see \cite{3}), non-smooth functionals (see \cite{ylp,iann,patrik1}), and singular equations with critical terms (see \cite{34}).\\

 As far as we know, there is only one paper (see \cite{6}) devoted to the result of Brezis and Nirenberg in the Orlicz case. Precisely, in \cite{6}, F. Fang and Z. Tan proved a boundedness regularity result and established the relation between the $C^1(\overline{\Omega})$ and $W^{1,G}_0(\Omega)$ minimizers for an Orlicz problem with Dirichlet boundary condition. Since our problem $(\ref{P})$ is with Robin boundary condition,  many approaches  used in \cite{6} don't work.\\

 The main novelty of our work is the study of the boundedness regularity for weak solutions of problem  $(\ref{P})$ and the relationship between the Orlicz-Sobolev and H\"older local minimizers for the energy functional of problem $(\ref{P})$. The non-homogeneity of the $g$-Laplacian operator brings us several difficulties in order to get the boundedness of a weak solution to the Robin Problem $(\ref{P})$.\\

 This paper is organized as follows. In Section 2, we recall the basic properties of
the Orlicz Sobolev spaces and the Orlicz Laplacian operator, and we state the
main hypotheses on the data of our problem. Section 3 deals with two
regularity results. In the first we prove that every weak solution of problem $(\ref{P})$ belongs to $L^{s}(\Omega)$, for all $1\leq s<\infty$. In the second we show that every solution of problem $(\ref{P})$ is bounded. In the last Section, we establish the relationship between the local $C^1(\overline{\Omega})$-minimizer and the local $W^{1,G}(\Omega)$-minimizer for the corresponding energy functional.

\section{Preliminaries}

  To deal with problem $(\ref{P})$, we use the theory of Orlicz-Sobolev spaces since problem $(\ref{P})$
contains a non-homogeneous function $a(.)$ in the differential operator. Therefore, we start
with some basic concepts of Orlicz-Sobolev spaces, and we set the hypotheses on the non-linear term $f$. For more details on the Orlicz-Sobolev spaces see \cite{111,6,8,9990,5,11} and the references therein.\\

The function $a : (0,+\infty) \rightarrow (0,+\infty)$ is a function such that the mapping, defined by
$$
g(t):=
\left\lbrace
\begin{array}{lcc}
a(\vert t\vert)t, & \text{if}\ t\neq 0,\\
\ & \ \\
0 ,& \text{if}\ t= 0,
\end{array}
\right.
$$
is an odd, increasing homeomorphism from $\mathbb{R}$ onto itself. Let
$$G(t):=\int_{0}^{t}g(s)\ {\rm d}s,\ \ \forall\ t\in \mathbb{R},$$
$G$ is an $N$-function, i.e. Young function satisfying: $G$ is even, positive, continuous and  convex function. Moreover, $G(0)=0$, $\frac{G(t)}{t}\rightarrow 0$ as $t\rightarrow0$ and  $\frac{G(t)}{t}\rightarrow +\infty$ as $t\rightarrow+\infty$ (see \cite[Lemma 3.2.2, p. 128]{9990}).
\\

In order to construct an Orlicz-Sobolev space setting for problem $(\ref{P})$, we impose the following class of assumptions on $G$, $a$ and $g$:
\begin{enumerate}
\item[$(G)$]
\begin{enumerate}
\item[$(g_1)$]: $a(t)\in C^{1}(0,+\infty),\ a(t)>0$ and $a(t)$ is an increasing function for $t>0.$
\item[$(g_2)$]: $1<p<g^-:=\inf\limits_{t>0}\frac{g(t)t}{G(t)}\leq g^{+}:=\sup\limits_{t>0}\frac{g(t)t}{G(t)}< N.$
\item[$(g_3)$]: $0<g^--1=a^-:=\inf\limits_{t>0}\frac{g^{'}(t)t}{g(t)}\leq g^+-1=a^{+}:=\sup\limits_{t>0}\frac{g^{'}(t)t}{g(t)}.$
\item[$(g_4)$]: $\displaystyle{\int_1^{+\infty}\frac{G^{-1}(t)}{t^{\frac{N+1}{N}}}{\rm d}t=\infty}$ and $\displaystyle{\int_0^{1}\frac{G^{-1}(t)}{t^{\frac{N+1}{N}}}{\rm d}t<\infty}$.
\end{enumerate}
\end{enumerate}
The conjugate $N$-function of $G$, is defined by
$$\tilde{G}(t)=\int_{0}^{ t} \tilde{g}(s)\ {\rm d}s,$$
where  $\tilde{g}: \mathbb{R}\rightarrow\mathbb{R}$ is given by
$\tilde{g}(t)=\sup\{s:\ g(s)\leq t\}$. If $g$ is continuous on $\mathbb{R}$, then $\tilde{g}(t)=g^{-1}(t)$ for all $t\in \mathbb{R}.$ Moreover, we have
\begin{equation}\label{1119}
st\leq G(s)+\tilde{G}(t),
\end{equation}
which is known as the Young inequality. Equality in $(\ref{1119})$ hol{\rm d}s  if and only if either $t=g(s)$ or $s=\tilde{g}(t)$. In our case, since $g$ is continuous, we have
 $$\tilde{G}(t)=\int_{0}^{ t}g^{-1}(s)\ {\rm d}s.$$
 The functions $G$ and $\tilde{G}$ are complementary $N$-functions.\\

We say that $G$  satisfies the $\Delta _2$-condition, if there exists $C>0$, such that
\begin{equation}\label{delta}
G(2t) \leq CG(t),\ \text{for all}\  t > 0.
\end{equation}
We want to remark that assumption $(g_2)$ and $(\ref{delta})$ are equivalent (see \cite[Theorem 3.4.4, p. 138]{9990} and \cite{8}).\\

If $G_1$ and $G_2$ are two $N$-functions, we say that $G_1$ grow essentially more slowly than $G_2$  $(G_1\prec\prec G_2$ in symbols$)$, if and only if for every positive constant $k$, we have
\begin{equation}
\lim_{t\rightarrow+\infty}\frac{G_1(kt)}{G_2(t)}=0.
\end{equation}

Another important function related to the $N$-function $G,$ is the Sobolev conjugate function $G_{*}$  defined by 
$$G_{*}^{-1}(t)=\int_{0}^{t}\frac{G^{-1}(s)}{s^{\frac{N+1}{N}}}\ {\rm d}s, \ t>0$$
(see \cite[Definition 7.2.1, p. 352]{9990}).\\
If $G$ satisfies the $\Delta _2$-condition, then $G_*$ also satisfies the $\Delta _2$-condition. Namely, there exist $g^-_*=\frac{Ng^-}{N-g^-}$ and $g^+_*=\frac{Ng^+}{N-g^+}$ such that
 \begin{equation}\label{delta2}
g^+<g^-_*:=\inf\limits_{t>0}\frac{g_*(t)t}{G_*(t)}\leq \frac{g_*(t)t}{G_*(t)}\leq g^{+}_*:=\sup\limits_{t>0}\frac{g_*(t)t}{G_*(t)}<+\infty,\ \text{for all}\ t>0
\end{equation}
(see \cite[Lemma 2.4, p. 240]{8}).\\

The Orlicz space $L^{G}(\Omega)$ is the vectorial space of measurable functions $u:\Omega\rightarrow\mathbb{R}$ such that
$$\rho(u)=\int_{\Omega}G(\vert u(x)\vert)\ {\rm d}x < \infty.$$
  \\
$L^{G}(\Omega)$ is a Banach space under the Luxemburg norm
$$\Vert u\Vert_{(G)}=\inf \left\lbrace  \lambda >0\ : \ \rho(\frac{u}{\lambda}) \leq 1\right\rbrace. $$
For Orlicz spaces, the H\"older inequality rea{\rm d}s as follows
$$\int_{\Omega}uv{\rm d}x\leq \Vert u\Vert_{(G)}\Vert v\Vert_{(\tilde{G})},\ \ \text{for all}\ u\in L^{G}(\Omega)\ \ \text{and}\ u\in L^{\tilde{G}}(\Omega).$$

Next, we introduce the Orlicz-Sobolev space. We denote by
$W^{1,G}(\Omega)$ the Orlicz-Sobolev space defined by
$$W^{1,G}(\Omega):=\bigg{\{}u\in L^{G}(\Omega):\ \frac{\partial u}{\partial x_{i}}\in L^{G}(\Omega),\ i=1,...,N\bigg{\}}.$$
$W^{1,G}(\Omega)$ is a Banach space with respect to the norm
$$\|u\|_G=\|u\|_{(G)}+ \|\nabla u\|_{(G)}.$$
Another equivalent norm is
$$\Vert u\Vert=\inf \left\lbrace  \lambda >0\ : \ \mathcal{K}(\frac{u}{\lambda}) \leq 1\right\rbrace,$$
where
\begin{equation}\label{mod}
\mathcal{K}(u)=\int_{\Omega}G(\vert \nabla u(x)\vert) {\rm d}x+\int_{\Omega}G(\vert u(x)\vert)\ {\rm d}x.
\end{equation}

If $G$ and its complementary function $\tilde{G}$ satisfied the $\Delta _2$-condition, then  $W^{1,G}(\Omega)$ is Banach, separable and reflexive space. For that, in our work, we also assume that $\tilde{G}$ satisfies the $\Delta _2$-condition.\\

In the sequel, we give general results related to the $N$-function and the Orlicz, Orlicz-Sobolev spaces.

\begin{lemma}\label{lem12}(see \cite{11}). Let $G$  and $H$ be $N$-functions, such that $H$ grows essentially more slowly than $G_*$ (where $G_*$ is the Sobolev conjugate function of $G$).
\begin{enumerate}
\item[$(1)$]If $\displaystyle{\int_1^{+\infty}\frac{G^{-1}(t)}{t^{\frac{N+1}{N}}}{\rm d}t=\infty}$ and $\displaystyle{\int_0^{1}\frac{G^{-1}(t)}{t^{\frac{N+1}{N}}}{\rm d}t<\infty}$, then the embedding $W^{1,G}(\Omega)\hookrightarrow L^{H}(\Omega)$ is compact and the embedding $W^{1,G}(\Omega)\hookrightarrow L^{G_*}(\Omega)$ is continuous.
\item[$(2)$] If $\displaystyle{\int_1^{+\infty}\frac{G^{-1}(t)}{t^{\frac{N+1}{N}}{\rm d}t}<\infty}$, then the embedding $W^{1,G}(\Omega)\hookrightarrow L^{H}(\Omega)$ is compact and the embedding $W^{1,G}(\Omega)\hookrightarrow L^{\infty}(\Omega)$ is continuous.
\end{enumerate}
\end{lemma}
\begin{lemma}\label{lem1}(see \cite{8})\\
Let $G$ be an $N$-function satisfying $(g_1)-(g_3)$ such that $\displaystyle{G(t)=\int^{ t}_0 g( s)\ {\rm d}s=\int^{ t}_0 a(\vert s\vert)s\ {\rm d}s}$. Then
\begin{enumerate}
\item[$(1)$] $\min\lbrace t^{g^-}, t^{g^+}\rbrace G(1)\leq G(t)\leq \max\lbrace t^{g^-}, t^{g^+}\rbrace G(1)$,  for all $0<t$;
\item[$(2)$] $\min\lbrace t^{g^--1}, t^{g^+-1}\rbrace g(1)\leq g(t)\leq \max\lbrace t^{g^--1}, t^{g^+-1}\rbrace g(1)$,  for all $0<t$;
\item[$(3)$] $\min\lbrace t^{g^--2}, t^{g^+-2}\rbrace a(1)\leq a(t)\leq \max\lbrace t^{g^--2}, t^{g^+-2}\rbrace a(1)$,  for all $0<t$;
\item[$(4)$] $\min\lbrace t^{g^-}, t^{g^+}\rbrace G(z)\leq G(tz)\leq \max\lbrace t^{g^-}, t^{g^+}\rbrace G(z)$,  for all $0<t$ and $z\in\mathbb{R}$;
\item[$(5)$] $\min\lbrace t^{g^--1}, t^{g^+-1}\rbrace g(z)\leq g(tz)\leq \max\lbrace t^{g^--1}, t^{g^+-1}\rbrace g(z)$,  for all $0<t$ and $z\in\mathbb{R}$;
\item[$(6)$] $\min\lbrace t^{g^--2}, t^{g^+-2}\rbrace a(\vert \eta\vert)\leq a(\vert t\eta\vert)\leq \max\lbrace t^{g^--2}, t^{g^+-2}\rbrace a(\vert \eta\vert)$,  for all $0<t$ and $\eta \in \mathbb{R}^N$.
\end{enumerate}
\end{lemma}
\begin{lemma}\label{lem13}(See \cite{8}). Let $G$ be an $N$-function satisfying $(g_2)$ such that $\displaystyle{G(t)=\int^{t}_0 g( s)\ {\rm d}s}$. Then
\begin{enumerate}
\item[$(1)$] if $\Vert u\Vert_{(G)}<1$ then $\Vert u\Vert_{(G)}^{g^+}\leq \rho(u)\leq \Vert u\Vert_{(G)}^{g^-}$;
\item[$(2)$] if $\Vert u\Vert_{(G)}\geq 1$ then $\Vert u\Vert_{(G)}^{g^-}\leq \rho(u)\leq \Vert u\Vert_{(G)}^{g^+}$;
\item[$(3)$] if $\Vert u\Vert<1$ then $\Vert u\Vert^{g^+}\leq \mathcal{K}(u)\leq \Vert u\Vert^{g^-}$;
\item[$(4)$] if $\Vert u\Vert\geq 1$ then $\Vert u\Vert^{g^-}\leq \mathcal{K}(u)\leq \Vert u\Vert^{g^+}$.
\end{enumerate}

\end{lemma}


\begin{lemma}\label{lem001}
 Assume that $\Omega$ is a bounded domain with smooth boundary $\partial \Omega$. Then the embedding $W^{1,p}(\Omega)\hookrightarrow L^r(\Omega)$ is compact provided $1 \leq r < p^*$, where
$p^*=\frac{Np}{N-p}$ if $p < N$ and $p^* := +\infty$ otherwise.
\end{lemma}
\begin{lemma}\label{lem002}
Assume that $\Omega$ is a bounded domain and has a Lipschitz boundary  $\partial\Omega$. Then the embedding $W^{1,p}(\Omega)\hookrightarrow L^r(\partial\Omega)$ is compact provided $1 \leq r < p^*$.
\end{lemma}
\begin{thm}\label{rem5}
The Orlicz-Sobolev space $W^{1,G}(\Omega)$ is continuously and compactly embedded in the classical Lebesgue spaces $L^r(\Omega)$ and $L^r(\partial\Omega)$ for all $1 \leq r < g^-_*$.
\end{thm}
\begin{proof}
By help of the assumption $(g_2)$, the Orlicz-Sobolev space $W^{1,G}(\Omega)$ is continuously embedded in the classical Sobolev space $W^{1,g^-}(\Omega)$. In light of Lemmas \ref{lem001} and \ref{lem002}, we deduce that $W^{1,g^-}(\Omega)$ is compactly embedded in $L^r(\Omega)$ and $L^r(\partial\Omega)$ for all $1 \leq r < g^-_*$. Hence, $W^{1,G}(\Omega)$ is continuously and compactly embedded in the classical Lebesgue space $L^r(\Omega)$ and $L^r(\partial\Omega)$ for all $1 \leq r < g^-_*$.
\end{proof}

\begin{lemma}\label{lem3}\cite[Lemma 3.2, p. 354]{6}\\
\begin{enumerate}
\item[$(1)$] If $a(t)$ is increasing for $t>0$, there exists constant $d_1$ depending on $g^-,\ g^+$, such that
\begin{equation} \label{111}
\vert a(\vert \eta\vert)\eta -a(\vert \xi\vert)\xi\vert \leq d_1\vert \eta -\xi\vert a(\vert \eta\vert +\vert \xi\vert),
\end{equation}
for all $\eta,\ \xi\ \in \mathbb{R}^N$.
\item[$(2)$] If $a(t)$ is decreasing for $t>0$, there exists constant $d_2$ depending on $g^-,\ g^+$, such that
\begin{equation} \label{112}
\vert a(\vert \eta\vert)\eta -a(\vert \xi\vert)\xi\vert \leq d_2 g(\vert \eta -\xi\vert),
\end{equation}
for all $\eta,\ \xi\ \in \mathbb{R}^N$.
\end{enumerate}
\end{lemma}

\begin{lemma}\label{lem444}
Let $G$ be an $N$-function satisfying $(g_1)-(g_3)$ such that $\displaystyle{G(t)=\int^{ t}_0 g( s)\ {\rm d}s=\int^{ t}_0 a(\vert s\vert)s\ {\rm d}s}$. Then for every $\xi,\eta\in\mathbb{R}^N$, we have
$$\langle a(\vert \eta\vert)\eta-a(\vert \xi\vert)\xi,\eta-\xi\rangle_{\mathbb{R}^N}\geq 0$$
where $\langle.\rangle_{\mathbb{R}^N}$ is the inner product on $\mathbb{R}^N$.
\end{lemma}
\begin{proof}[Proof]
Let $\eta,\xi\in\mathbb{R}^N$. Since $G$ is  convex, we have
$$G(\vert \eta\vert)\leq G\left( \left| \frac{\eta+\xi}{2}\right|\right)+\langle a(\vert \eta\vert)\eta,\frac{\eta-\xi}{2}\rangle_{\mathbb{R}^N} $$
and
$$G(\vert \xi\vert)\leq G\left( \left| \frac{\eta+\xi}{2}\right|\right)+\langle a(\vert \xi\vert)\xi,\frac{\xi-\eta}{2}\rangle_{\mathbb{R}^N} . $$
Adding the above two relations, we find that
\begin{align}\label{hlel}
\frac{1}{2}\langle a(\vert \eta\vert)\eta-a(\vert \xi\vert)\xi,\eta-\xi\rangle_{\mathbb{R}^N} \geq G(\vert \eta\vert)+G(\vert \xi\vert)-2G\left( \left| \frac{\eta +\xi}{2}\right| \right) \ \ \text{for all}\ \ \eta,\xi\in \mathbb{R}^N.
\end{align}
On the other hand, the convexity and the monotonicity of $G$ give
\begin{align}\label{miss}
G\left( \left| \frac{\eta +\xi}{2}\right| \right)\leq \frac{1}{2}\left[ G\left( \vert \eta \vert \right)+G\left( \vert \xi \vert \right)\right] \ \ \text{for all}\ \ \eta,\xi\in \mathbb{R}^N.
\end{align}
From $(\ref{hlel})$ and $(\ref{miss})$, we get
\begin{align*}
\langle a(\vert \eta\vert)\eta-a(\vert \xi\vert)\xi,\eta-\xi\rangle_{\mathbb{R}^N}\geq 0,\ \ \text{for all}\ \ \eta,\xi\in \mathbb{R}^N.
\end{align*}
The proof is now complete.
\end{proof}

\begin{definition}\label{def0}(See \cite{5})\\
We say that $u\in W^{1,G}(\Omega)$ is a weak solution for problem $(\ref{P})$ if
\begin{align}\label{O.O}
& \int_{\Omega} a(\vert\nabla u\vert)\nabla u.\nabla v {\rm d}x +\int_{\partial\Omega}b(x)\vert u\vert^{p-2}u v {\rm d}\gamma
 =\int_{\Omega}f(x,u)v{\rm d}x,\ \forall v\in W^{1,G}(\Omega)
\end{align}
where ${\rm d}\gamma$ is the measure on the boundary $\partial\Omega$.\\
The energy functional corresponding to problem  $(\ref{P})$ is  the $C^1$-functional  $J :W^{1,G}(\Omega)\rightarrow \mathbb{R}$
defined by
\begin{equation}\label{eq9870}
J(u)=\int_{\Omega}G(\vert\nabla u\vert) {\rm d}x +\frac{1}{p}\int_{\partial\Omega}b(x)\vert u\vert^{p} {\rm d}\gamma-\int_{\Omega}F(x,u){\rm d}x,
\end{equation}
for all $u\in W^{1,G}(\Omega)$. Where $\displaystyle{F(x, t) = \int_{0}^{t} f(x, s) {\rm d}s}$.\\
\end{definition}
\begin{definitions}\label{def333}\ \\
\begin{enumerate}
\item[$(1)$] We say that $u_0\in W^{1,G}(\Omega)$ is a local $C^1(\overline{\Omega})$-minimizer of $J$, if we can find $r_0>0$ such that
$$J(u_0)\leq J(u_0+v),\ \text{for all}\ v\in C^1(\overline{\Omega})\ \text{with}\ \Vert v\Vert_{C^1(\overline{\Omega})}\leq r_0.$$
\item[$(2)$] We say that $u_0\in W^{1,G}(\Omega)$ is a local $W^{1,G}(\Omega)$-minimizer of $J$, if we can find $r_1>0$ such that
$$J(u_0)\leq J(u_0+v),\ \text{for all}\ v\in W^{1,G}(\Omega)\ \text{with}\ \Vert v\Vert\leq r_1.$$
\end{enumerate}
\end{definitions}
Now, we set the assumption on the non-linear term $f$ as follows.
\begin{enumerate}
\item[$(H)$]  $f(x,0)=0$ and there exist  an odd increasing homomorphism $h\in C^{1}(\mathbb{R},\mathbb{R})$, and  a positive function $\widehat{a}(t)\in L^{\infty}(\Omega)$ such that
$$\vert f(x,t)\vert \leq \widehat{a}(x)(1+h(\vert t\vert)), \ \ \forall\ t\in \mathbb{R},\ \forall x\in \overline{\Omega}$$
and
$$G\prec\prec H\prec\prec G_*,$$
 $$\displaystyle{1<g^+<h^-:=\inf\limits_{t>0}\frac{h(t)t}{H(t)}\leq h^{+}:=\sup\limits_{t>0}\frac{h(t)t}{H(t)}\leq\frac{g^-_*}{g^-}},$$
 $$\displaystyle{1<h^--1:=\inf\limits_{t>0}\frac{h^{'}(t)t}{h(t)}\leq h^{+}-1:=\sup\limits_{t>0}\frac{h^{'}(t)t}{h(t)}},$$
 where
 $$H(t):=\int_0^t h(s)\ {\rm d}s,$$
 is an $N$-function.
\end{enumerate}

\begin{rem}\label{rem78}
Some assertions  in Lemma \ref{lem1} are remain valid for the $N$-function $H$ and the function $h$
\begin{enumerate}
\item[$(1)$] $\min\lbrace t^{h^-}, t^{h^+}\rbrace H(1)\leq H(t)\leq \max\lbrace t^{h^-}, t^{h^+}\rbrace H(1)$,  for all $0<t$;
\item[$(2)$] $\min\lbrace t^{h^--1}, t^{h^+-1}\rbrace h(1)\leq h(t)\leq \max\lbrace t^{h^--1}, t^{h^+-1}\rbrace h(1)$,  for all $0<t$;
\item[$(3)$] $\min\lbrace t^{h^-}, t^{h^+}\rbrace H(z)\leq H(tz)\leq \max\lbrace t^{h^-}, t^{h^+}\rbrace H(z)$,  for all $0<t$ and $z\in\mathbb{R}$;
\item[$(4)$] $\min\lbrace t^{h^--1}, t^{h^+-1}\rbrace h(z)\leq h(tz)\leq \max\lbrace t^{h^--1}, t^{h^+-1}\rbrace h(z)$,  for all $0<t$ and $z\in\mathbb{R}$.
\end{enumerate}
\end{rem}
The main results of this paper are:
\begin{thm}\label{thmC11}
Under the assumptions $(G)$ and $(H)$, if $u\in W^{1,G}(\Omega)$ is a non-trivial  weak solution of problem $(\ref{P})$, then $u\in L^{\infty}(\Omega)$ and $\Vert u\Vert_{\infty}\leq M=M(\Vert \widehat{a}\Vert_{\infty}, h(1), g^-,\vert \Omega\vert, \Vert u \Vert_{h^+})$.
\end{thm}
\begin{thm}\label{thmC12}
Under the assumptions $(G)$ and $(H)$, if $u_0\in W^{1,G}(\Omega)$ is a local $C^1(\overline{\Omega})$-minimizer of $J$,
then $u_0\in C^{1,\alpha}(\overline{\Omega})$ for some $\alpha\in (0,1)$ and $u_0$ is also a local $W^{1,G}(\Omega)$-minimizer of $J$.
\end{thm}
\section{Boundedness results for weak solutions of problem $(\ref{P})$}
In this section, by using the Moser iteration technique, we prove a result concerning the boundedness regularity for the problem $(\ref{P})$. Our method, inspired by the work of Gasi\`nski and Papageorgiou \cite{3}.\\

Considering the following problem
\begin{equation}\label{A}
\left\lbrace
\begin{array}{ll}
-\text{div}(\mathcal{A}(x,\nabla u))=\mathcal{B}(x,u),\ \text{in}\ \Omega\\
\ \\
\mathcal{A}(x,\nabla u).\nu+\psi(x,u)=0,\ \text{in}\ \partial\Omega
\end{array}
\right.
\tag{A}
\end{equation}
where $\Omega$ is a bounded subset of $\mathbb{R}^N(N\geq 3)$ with $C^2$-boundary, $\mathcal{A}:\Omega\times\mathbb{R}^N\rightarrow\mathbb{R}^N$, $\mathcal{B}:\Omega\times\mathbb{R}\rightarrow\mathbb{R}$ and $\psi:\partial\Omega\times\mathbb{R}\rightarrow\mathbb{R}$. We assume that problem $(\ref{A})$ satisfies the following growth conditions:
\begin{equation}\label{7546999882}
\mathcal{A}(x,\eta)\eta\geq  G(\vert \eta\vert),\ \ \text{for all}\ x\in \Omega\ \text{and}\ \eta\in \mathbb{R}^N,
\end{equation}
\begin{equation}\label{75469998822}
\mathcal{A}(x,\eta)\leq c_0 g(\vert \eta\vert)+c_1,\ \ \text{for all}\ x\in \Omega\ \text{and}\ \eta\in \mathbb{R}^N,
\end{equation}
\begin{equation}\label{754699988222}
\mathcal{B}(x,t)\leq c_2(1+h(\vert t\vert)),\ \ \text{for all}\ x\in \Omega\ \text{and}\ t\in \mathbb{R},
\end{equation}
\begin{equation}\label{7546999882222}
\psi(x,t)\geq 0,\ \ \text{for all}\ x\in \partial\Omega\ \text{and}\ t\in \mathbb{R}_+,
\end{equation}
 where $c_0,c_1,c_2$ are positive constant and $h$ is defined in assumption $(H)$.\\

We say that $u\in W^{1,G}(\Omega)$ is  a weak solution of problem $(\ref{A})$ if
\begin{equation}\label{4546787856}
\int_{\Omega}\mathcal{A}(x,\nabla u)\nabla v {\rm d}x+\int_{\partial\Omega}\psi(x,u)v{\rm d}\gamma=\int_{\Omega}\mathcal{B}(x,u)v{\rm d}x,\ \ \text{for all}\ v\in W^{1,G}(\Omega).
\end{equation}

Let us state the following  useful result
\begin{prop}\label{prop1}Suppose that $(G)$, $(H)$ and $(\ref{7546999882})-(\ref{7546999882222})$ are satisfied. Then, if $u \in W^{1,G}(\Omega)$ is a non-trivial weak solution  of problem $(\ref{A})$, $u$ belongs to $L^{s}(\Omega)$ for every $1\leq s<\infty$.
\end{prop}

\begin{proof}[Proof]
Let $u \in W^{1,G}(\Omega)$ be a non-trivial  weak solution of problem $(\ref{A})$, $u^{+}:=\max\lbrace u,0\rbrace\in W^{1,G}(\Omega)$ and $u^{-}:=\max\lbrace -u,0\rbrace\in W^{1,G}(\Omega)$. Since $u=u^+ - u^-$, without loss of generality we may assume that $u\geq 0$.\\
We set, recursively
$$p_{n+1}=\widehat{g} + \frac{\widehat{g}}{g^-}\left(\frac{p_n -h^+}{h^+}\right),\ \ \text{for all}\ n\geq 0,$$
such that
$$ p_0=\widehat{g} = g_*^-=\frac{Ng^-}{N-g^-}\ \ \ (\text{recall that}\ g^-\leq g^+<N).$$
It is clear that the sequence $\lbrace p_n\rbrace_{n\geq 0}\subseteq \mathbb{R}_+$ is increasing. Put
$\displaystyle{\theta_n=\frac{p_n -h^+}{h^+}>0}, \lbrace\theta_n\rbrace_{n\geq 0}$ is an increasing sequence.\\
Let
$$u_k=\min \lbrace u,k\rbrace\in W^{1,G}(\Omega)\cap L^{\infty}(\Omega), \text{for all}\ k\geq 1\ \ \ (\text{since }\ u_k\leq k,\ \ \text{for all}\ k\geq 1).$$
In $(\ref{4546787856})$, we act with $u_k^{\theta_n +1}\in W^{1,G}(\Omega)$, to obtain
\begin{align*}
\int_{\Omega} \mathcal{A}(x,\nabla u).\nabla u_k^{\theta_n +1}{\rm d}x &+\int_{\partial\Omega}\psi(x,u) u_k^{\theta_n +1} {\rm d}\gamma =\int_{\Omega}\mathcal{B}(x,u)u_k^{\theta_n +1}{\rm d}x.
\end{align*}
It follows, by conditions $(\ref{7546999882})$, $(\ref{7546999882222})$ and Lemma \ref{lem1}, that
\begin{align}\label{3}
(\theta_n +1)\int_{\lbrace\vert\nabla u_k\vert\leq 1\rbrace}u_k^{\theta_n}G(\vert\nabla u_k\vert) {\rm d}x &+ (\theta_n +1)G(1)\int_{\lbrace\vert\nabla u_k\vert> 1\rbrace}u_k^{\theta_n}\vert\nabla u_k\vert^{g^-} {\rm d}x\nonumber\\
& \leq  (\theta_n +1)\int_{\Omega}u_k^{\theta_n}G(\vert\nabla u_k\vert) {\rm d}x\nonumber\\
 & \leq (\theta_n +1)\int_{\Omega}u_k^{\theta_n}\left[ \mathcal{A}(x,\nabla u).\nabla u_k\right] {\rm d}x\nonumber \\
 & \leq \int_{\Omega} \mathcal{A}(x,\nabla u).\nabla u_k^{\theta_n +1}{\rm d}x\nonumber\\
& \leq \int_{\Omega}\mathcal{B}(x,u)u_k^{\theta_n +1}{\rm d}x.
\end{align}
Therefore
\begin{align}\label{333}
(\theta_n +1)G(1)\int_{\lbrace\vert\nabla u_k\vert> 1\rbrace}u_k^{\theta_n}\vert\nabla u_k\vert^{g^-} {\rm d}x& \leq \int_{\Omega}\mathcal{B}(x,u)u_k^{\theta_n +1}{\rm d}x,
\end{align}
this gives,
\begin{align}\label{334}
(\theta_n +1)G(1)\int_{\Omega}u_k^{\theta_n}\vert\nabla u_k\vert^{g^-} {\rm d}x & =
(\theta_n +1)G(1)\left[ \int_{\lbrace\vert\nabla u_k\vert> 1\rbrace}u_k^{\theta_n}\vert\nabla u_k\vert^{g^-} {\rm d}x+\int_{\lbrace\vert\nabla u_k\vert\leq 1\rbrace}u_k^{\theta_n}\vert\nabla u_k\vert^{g^-} {\rm d}x\right]\nonumber\\
 & \leq \int_{\Omega}\mathcal{B}(x,u)u_k^{\theta_n +1}{\rm d}x +(\theta_n +1)G(1)\int_{\lbrace\vert\nabla u_k\vert\leq 1\rbrace}u_k^{\theta_n}\vert\nabla u_k\vert^{g^-} {\rm d}x\nonumber\\
 & \leq \int_{\Omega}\mathcal{B}(x,u)u_k^{\theta_n +1}{\rm d}x +(\theta_n +1)G(1)\int_{\Omega}u_k^{\theta_n} {\rm d}x.
\end{align}
Thus
\begin{align}\label{33455}
\int_{\Omega}u_k^{\theta_n}\vert\nabla u_k\vert^{g^-} {\rm d}x & \leq \frac{1}{(\theta_n +1)G(1)} \int_{\Omega}\mathcal{B}(x,u)u_k^{\theta_n +1}{\rm d}x\nonumber \\
 & +\int_{\Omega}u_k^{\theta_n} {\rm d}x.
\end{align}
Since $\theta_n\leq p_n$, and by the continuous embedding $L^{p_n}(\Omega)\hookrightarrow L^{\theta_n}(\Omega)$, then
\begin{equation}\label{33505}
\int_{\Omega}u_k^{\theta_n} {\rm d}x\leq \vert \Omega\vert^{1-\frac{\theta_n}{p_n}}\Vert u_k\Vert^{\theta_n}_{p_n},\ \ \text{for all}\ k\geq 1.
\end{equation}
Combining $(\ref{33455})$ and $(\ref{33505})$, we infer that
\begin{equation}\label{3366}
\int_{\Omega}u_k^{\theta_n}\vert\nabla u_k\vert^{g^-} {\rm d}x\leq \frac{1}{(\theta_n +1)G(1)}\int_{\Omega}\mathcal{B}(x,u)u_k^{\theta_n +1}{\rm d}x+\vert \Omega\vert^{1-\frac{\theta_n}{p_n}}\Vert u_k\Vert^{\theta_n}_{p_n}.
\end{equation}
Let us observe that
\begin{align*}
\nabla u_k^{\frac{\theta_n +g^-}{g^-}} & =\nabla u_k^{(\frac{\theta_n}{g^-}+1)}= (\frac{\theta_n}{g^-}+1)u_k^{\frac{\theta_n}{g^-}}\nabla u_k
\end{align*}
and
\begin{align*}
\left| \nabla u_k^{\frac{\theta_n +g^-}{g^-}}\right| ^{g^-} =  \left( \frac{\theta_n}{g^-}+1\right) ^{g^-}u_k^{\theta_n}\left|\nabla u_k\right|^{g^-}.
\end{align*}
Integrating  over $\Omega$, we get
\begin{align}\label{4}
\int_{\Omega} \left| \nabla u_k^{\frac{\theta_n +g^-}{g^-}}\right|^{g^-}{\rm d}x & = \left( \frac{\theta_n}{g^-}+1\right)^{g^-}\int_{\Omega}u_k^{\theta_n}\left|\nabla u_k\right|^{g^-}{\rm d}x.
\end{align}
Putting together $(\ref{3366})$ and $(\ref{4})$, we conclude that
\begin{align}\label{5}
\int_{\Omega} \left| \nabla u_k^{\frac{\theta_n +g^-}{g^-}}\right|^{g^-}{\rm d}x & \leq \left( \frac{\theta_n}{g^-}+1\right)^{g^-}\left[ \frac{1}{(\theta_n +1)G(1)}\int_{\Omega}\mathcal{B}(x,u)u_k^{\theta_n +1}{\rm d}x+\vert \Omega\vert^{1-\frac{\theta_n}{p_n}}\Vert u_k\Vert^{\theta_n}_{p_n}\right]\nonumber\\
& \leq \left( \theta_n+1\right)^{g^-}\left[ \frac{1}{(\theta_n +1)G(1)}\int_{\Omega}\mathcal{B}(x,u)u_k^{\theta_n +1}{\rm d}x+\vert \Omega\vert^{1-\frac{\theta_n}{p_n}}\Vert u_k\Vert^{\theta_n}_{p_n}\right]\nonumber\\
&\leq C_0\left( \int_{\Omega}\mathcal{B}(x,u)u_k^{\theta_n +1}{\rm d}x+\left(  1+\Vert u_k\Vert^{p_n}_{p_n}\right) \right),
\end{align}
where $\displaystyle{C_0=\left( \theta_n+1\right)^{g^-}\left( \frac{1}{(\theta_n +1)G(1)}+\vert \Omega\vert^{1-\frac{\theta_n}{p_n}}\right) > 0}.$\\
On the other side, using the condition $(\ref{754699988222})$ and  Remark \ref{rem78}, we see that
\begin{align}\label{6}
\int_{\Omega}\mathcal{B}(x,u)u_k^{\theta_n +1}{\rm d}x & \leq  c_2\int_{\Omega}\left( 1+ h(\vert u\vert)\right)u_{k}^{\theta_n +1}{\rm d}x\nonumber\\
& \leq  c_2\int_{\Omega}\left(  1+ h(1)\max\lbrace\vert u\vert^{h^--1},\vert u\vert^{h^+-1}\rbrace \right)u_k^{\theta_n +1}{\rm d}x\nonumber\\
& \leq  c_2\left( \Vert u_k\Vert_{\theta_n+1}^{{\theta_n+1}}+ h(1)\int_{\Omega}\max\lbrace\vert u\vert^{h^--1},\vert u\vert^{h^+-1}\rbrace u_k^{\theta_n +1}{\rm d}x\right) \nonumber\\
& \leq  c_2\left[ \Vert u_k\Vert_{\theta_n+1}^{{\theta_n+1}}+ h(1)\left( \int_{\lbrace u\leq 1\rbrace} u^{h^--1} u_k^{\theta_n +1}{\rm d}x+\int_{\lbrace u>1\rbrace} u^{h^+-1} u_k^{\theta_n +1}{\rm d}x\right)\right]  \nonumber\\
& \leq  c_2\left[ \left( 1+h(1)\right) \Vert u_k\Vert_{\theta_n+1}^{{\theta_n+1}}+ h(1)\int_{\Omega} u^{h^+-1} u_k^{\theta_n +1}{\rm d}x\right]  \nonumber\\
& \leq c_2\left[  \left( 1+h(1)\right) \Vert u_k\Vert_{\theta_n+1}^{{\theta_n+1}}+ h(1)\Vert u\Vert_{h^+}^{h^+-1}\Vert u_k\Vert^{\theta_n +1}_{(\theta_n +1)h^+}\right]  \ \ (\text{H\"older with}\ h^+\ \text{and}\ (h^+)^{'}=\frac{h^+}{h^+-1})\nonumber\\
& = c_2\left[ \left( 1+h(1)\right) \Vert u_k\Vert_{\theta_n+1}^{{\theta_n+1}}+ h(1)\Vert u\Vert_{h^+}^{h^+-1}\Vert u_k\Vert^{\theta_n +1}_{p_n} \right] \nonumber\\
& \leq c_2\left[ \left( 1+h(1)\right)\vert \Omega\vert^{1-\frac{\theta_n+1}{p_n}} \Vert u_k\Vert^{\theta_n+1}_{{p_n}}+ h(1)\Vert u\Vert_{h^+}^{h^+-1}\Vert u_k\Vert^{\theta_n +1}_{p_n}\right]\ \ (\text{since}\ L^{p_n}(\Omega)\hookrightarrow L^{\theta_n+1}(\Omega))  \nonumber\\
& \leq c_2\left[  \left( 1+h(1)\right)\vert \Omega\vert^{1-\frac{1}{h^+}}+h(1)\Vert u\Vert_{h^+}^{h^+-1}\right] \Vert u_k\Vert_{p_n}^{\theta_n+1}\ \ (\text{since}\ \frac{\theta_n+1}{p_n}=\frac{1}{h^+})\nonumber\\
& \leq C_1\left(1+ \Vert u_k\Vert_{p_n}^{p_n}\right),
\end{align}
where $\displaystyle{C_1=c_2\left[  \left( 1+h(1)\right)\vert \Omega\vert^{1-\frac{1}{h^+}}+h(1)\Vert u\Vert_{h^+}^{h^+-1}\right] >0}$.  In $(\ref{6})$, we used the fact that $\theta_n+1<(\theta_n+1)h^+=p_n$.\\
 Using $(\ref{5})$ and  $(\ref{6})$, we find
\begin{align}\label{ddhlel}
\int_{\Omega}\left| \nabla u_k^{\frac{\theta_n +g^-}{g^-}}\right| ^{g^-}{\rm d}x+\int_{\Omega}\left|  u_k^{\frac{\theta_n +g^-}{g^-}}\right| ^{g^-}{\rm d}x &\leq C_2 \left(1+ \Vert u_k\Vert^{p_n}_{p_n} \right)+\int_{\Omega}\left|  u_k^{\frac{\theta_n +g^-}{g^-}}\right| ^{g^-}{\rm d}x\nonumber\\
& \leq C_2 \left(1+ \Vert u_k\Vert^{p_n}_{p_n} \right)+\vert \Omega\vert^{1-\frac{\theta_n+g^-}{p_n}}\Vert u_k\Vert^{\theta_n +g^-}_{p_n}\nonumber\\
& \leq \left( C_2+\vert \Omega\vert^{1-\frac{\theta_n+g^-}{p_n}}\right)  \left(1+ \Vert u_k\Vert^{p_n}_{p_n} \right)\nonumber\\
&  = C_3 \left(1+ \Vert u_k\Vert^{p_n}_{p_n} \right),
\end{align}
where $C_2=C_0(C_1+1)$ and $\displaystyle{C_3=C_2+\vert \Omega\vert^{1-\frac{\theta_n+g^-}{p_n}}}.$\\
The inequality $(\ref{ddhlel})$ gives

\begin{equation}\label{7}
 \left\lVert u_k^{\frac{\theta_n +g^-}{g^-}}\right\rVert^{g^-}_{W^{1,g^-}(\Omega)} \leq C_3 \left(1+ \Vert u_k\Vert^{p_n}_{p_n} \right).
\end{equation}
Recall that $p_{n+1}=\widehat{g}+\frac{\widehat{g}}{g^-}\theta_n$ and so
\begin{equation}\label{4787545455456}
\frac{\theta_n+g^-}{g^-}=\frac{p_{n+1}}{\widehat{g}}.
\end{equation}
Since $g^-<\widehat{g}=\frac{Ng^-}{N-g^-}=g^-_*$, then the embedding $W^{1,g^-}(\Omega)\hookrightarrow L^{\widehat{g}}(\Omega)$ is continuous.\\ Hence, there is $C_4>0$ such that
\begin{equation}\label{dddhlel}
\left\lVert u_k^{\frac{\theta_n +g^-}{g^-}}\right\rVert^{g^-}_{\widehat{g}}\leq C_4\left\lVert u_k^{\frac{\theta_n +g^-}{g^-}}\right\rVert^{g^-}_{W^{1,g^-}(\Omega)}.
\end{equation}
Combining $(\ref{7})$, $(\ref{4787545455456})$ and $(\ref{dddhlel})$, we obtain
\begin{equation}\label{8}
\Vert u_k\Vert^{\frac{p_{n+1}}{\widehat{g}}g^-}_{p_{n+1}}\leq C_5 \left(1+ \Vert u_k\Vert^{p_n}_{p_n} \right),
\end{equation}
where $C_5=C_4C_3.$ Next, let $k\rightarrow+\infty$ in $(\ref{8})$ and applying the  monotone convergence theorem, we find that
\begin{equation}\label{9}
\Vert u\Vert^{\frac{p_{n+1}}{\widehat{g}}g^-}_{p_{n+1}}\leq C_5 \left(1+ \Vert u\Vert^{p_n}_{p_n} \right).
\end{equation}
Since $p_0=\widehat{g}$ and the embeddings $W^{1,G}(\Omega)\hookrightarrow W^{1,g^-}(\Omega)\hookrightarrow  L^{\widehat{g}}(\Omega)$ are continuous, from $(\ref{9})$, we get
\begin{equation}\label{10}
u\in L^{p_n}(\Omega),\ \ \text{for all}\ n\geq 0.
\end{equation}
Note that $p_n\rightarrow +\infty$ as $n\rightarrow +\infty.$ Indeed, suppose that the sequence $\lbrace p_n\rbrace_{n\geq 0}\subseteq [\widehat{g}, +\infty)$ is bounded. Then we have $p_n\longrightarrow \widehat{p}\geq \widehat{g}$ as $n\rightarrow +\infty.$ By definition we have
$$p_{n+1}=\widehat{g}+\frac{\widehat{g}}{g^-}\left(\frac{p_n-h^+}{h^+}\right)\ \ \text{for all}\ n\geq 0,$$
with $p_0=\widehat{g}$, so
$$\widehat{p}=\widehat{g}+\frac{\widehat{g}}{g^-}\left(\frac{\widehat{p}-h^+}{h^+}\right),$$
thus $$0\leq\widehat{p}\left( \frac{\widehat{g}}{g^-h^+}-1\right)=\widehat{g}\left( \frac{1}{g^-}-1\right)<0$$
which gives us a contradiction since $g^-h^+\leq \widehat{g}=g^-_*$ (see  assumption $(H)$).
Recall that for any measurable function $u:\Omega\longrightarrow\mathbb{R},$ the set
$$S_u=\left\lbrace p\geq 1: \Vert u\Vert_p<+\infty\right\rbrace $$
is an interval. Hence, $S_u=[1,+\infty)$ (see $(\ref{10})$) and
\begin{equation}\label{11}
u\in L^s(\Omega),\ \ \text{for all}\ s\geq 1.
\end{equation}
This en{\rm d}s the proof.
\end{proof}
In the following, we prove that, if $u\in W^{1,G}(\Omega)$ is a weak solution  of problem $(\ref{A})$ such that $u\in L^s(\Omega)$ for all $1\leq s<\infty$, then $u$ is a bounded function.

\begin{prop}\label{prop2}
Assume that $(G)$, $(H)$ and $(\ref{7546999882})-(\ref{7546999882222})$ hold. Let $u\in W^{1,G}(\Omega)$ be a non-trivial  weak solution of problem $(\ref{A})$ such that $u\in L^{s}(\Omega)$ for all $1\leq s<\infty$, then $u\in L^{\infty}(\Omega)$ and $\Vert u\Vert_{\infty}\leq M=M(c_2, h(1), g^-,\vert \Omega\vert, \Vert u \Vert_{h^+})$.
\end{prop}
\begin{proof}[Proof]
Let $u \in W^{1,G}(\Omega)$ be a non-trivial  weak solution of problem $(\ref{A})$, $u^{+}:=\max\lbrace u,0\rbrace\in W^{1,G}(\Omega)$ and $u^{-}:=\max\lbrace -u,0\rbrace\in W^{1,G}(\Omega)$. Since $u=u^+ - u^-$, we may assume without loss of generality that $u\geq 0$.\\
Let $\sigma_0=\widehat{g}=g_*^-=\frac{Ng^-}{N-g^-}$ and we define by a recursively way
$$\sigma_{n+1}=\left(\frac{\sigma_n }{h^+}-1+g^-\right)\frac{\widehat{g}}{g^-},\ \ \text{for all}\ n\geq 0.$$
 We have that the sequence $\displaystyle{\lbrace \sigma_n\rbrace_{n\geq 0}\subseteq [\widehat{g}, +\infty)}$ is increasing and $\sigma_n\longrightarrow+\infty$ as $n\rightarrow+\infty$.
Arguing as in the proof of Proposition \ref{prop1}, with  $\theta_n=\frac{\sigma_n }{h^+}-1$ and $u_k^{\theta_n+1}\in W^{1,G}(\Omega)\cap L^\infty(\Omega)$ as a test function in $(\ref{4546787856})$. So, we find the following estimation
\begin{align}\label{588}
\int_{\Omega} \left| \nabla u_k^{\frac{\theta_n +g^-}{g^-}}\right|^{g^-}{\rm d}x & \leq \left( \theta_n+1\right)^{g^-}\left[ \frac{1}{(\theta_n +1)G(1)}\int_{\Omega}\mathcal{B}(x,u)u_k^{\theta_n +1}{\rm d}x+\int_{\Omega} u_k^{\theta_n}{\rm d}x\right]
\end{align}

Using the assumption $(\ref{754699988222})$, $(\ref{11})$,  Remark \ref{rem78} and H\"older inequality (with $h^+$ and $(h^+)^{'}=\frac{h^+}{h^+-1}$), we deduce that
\begin{align}\label{7888}
\int_{\Omega}\mathcal{B}(x,u)u_k^{\theta_n +1}{\rm d}x & =\int_{\Omega}\mathcal{B}(x,u)u_k^{\frac{\sigma_n }{h^+}}{\rm d}x\nonumber\\
 &\leq c_2\int_{\Omega}\left( 1+h(1)\max\left\lbrace u^{h^--1},u^{h^+-1}\right\rbrace \right)u_k^{\frac{\sigma_n }{h^+}}{\rm d}x\nonumber\\
&\leq c_2\int_{\Omega}\left( (1+h(1))+h(1)u^{h^+-1}\right)u_k^{\frac{\sigma_n }{h^+}}{\rm d}x\ (\ \text{since}\ h^-\leq h^+)\nonumber\\
&\leq c_2\left[(1+h(1))\int_{\Omega} u_k^{\frac{\sigma_n }{h^+}}{\rm d}x+ h(1)\int_{\Omega}u^{h^+-1}u_k^{\frac{\sigma_n }{h^+}}{\rm d}x\right]\nonumber \\
&\leq c_2\left[(1+h(1))\Vert u_k\Vert^{\frac{\sigma_n }{h^+}}_{\frac{\sigma_n }{h^+}}+ h(1)\left( \int_{\Omega}u^{h^+}{\rm d}x\right)^{\frac{h^+-1}{h^+}}\left( \int_{\Omega}u_k^{\sigma_n}{\rm d}x\right)^{\frac{1}{h^+}}\right]\nonumber\\
&\leq c_2\left((1+h(1))\vert \Omega\vert^{1-\frac{1}{h^+}}\Vert u_k\Vert^{\frac{\sigma_n }{h^+}}_{\sigma_n}+ h(1)\Vert u\Vert_{h^+}^{h^+-1}\Vert u_k\Vert_{\sigma_n}^{\frac{\sigma_n }{h^+}}\right)\ (\ \text{since}\ L^{\sigma_n}(\Omega)\hookrightarrow L^{\frac{\sigma_n }{h^+}}(\Omega)\ )\nonumber \\
&\leq c_2\left((1+h(1))\vert \Omega\vert^{1-\frac{1}{h^+}}+ h(1)\Vert u\Vert_{h^+}^{h^+-1}\right)\Vert u_k\Vert_{\sigma_n}^{\frac{\sigma_n }{h^+}}\nonumber \\
&\leq C_6\Vert u_k\Vert_{\sigma_n}^{\frac{\sigma_n }{h^+}}
\end{align}
for all $n\in\mathbb{N}$, where $\displaystyle{C_6=c_2\left((1+h(1))\vert \Omega\vert^{1-\frac{1}{h^+}}+ h(1)\Vert u\Vert_{h^+}^{h^+-1}\right)}.$\\
Using the fact that $L^{\sigma_n}(\Omega)\hookrightarrow L^{\frac{\sigma_n }{h^+}-1}(\Omega)$, we obtain
\begin{align}\label{888}
\int_{\Omega} u_k^{\theta_n}{\rm d}x &= \int_{\Omega} u_k^{\frac{\sigma_n }{h^+}-1}{\rm d}x\nonumber\\
& \leq \vert\Omega\vert^{1-\frac{1}{h^+}+\frac{1}{\sigma_n}}\Vert u_k\Vert_{\sigma_n}^{\frac{\sigma_n }{h^+}-1}\nonumber\\
&=C_7\Vert u_k\Vert_{\sigma_n}^{\frac{\sigma_n }{h^+}-1},\ \ \text{for all}\ n\in\mathbb{N},
\end{align}
 where $\displaystyle{C_7(n)=\vert\Omega\vert^{1-\frac{1}{h^+}+\frac{1}{\sigma_n}}}$. \\
By H\"older’s inequality (with exponents $h^+$ and $(h^+)^{'}=h^+/(h^+-1)$) and the embedding $L^{h^+}(\Omega)\hookrightarrow L^{\frac{h^+(g^--1)}{h^+-1}}(\Omega)$ (since $g^-< h^+$), we infer that
\begin{align}\label{455169965}
\int_{\Omega} \left|  u_k^{\frac{\theta_n +g^-}{g^-}}\right|^{g^-}{\rm d}x  &=\int_{\Omega}   u_k^{\theta_n +1}  u_k^{g^--1}{\rm d}x\nonumber\\
&\leq \int_{\Omega}   u_k^{\theta_n +1}  u^{g^--1}{\rm d}x\  (\text{since}\ u_k\leq u,\ \text{for all}\ k\geq 1)\nonumber\\
& \leq \left( \int_{\Omega} u^{\frac{h^+(g^--1)}{h^+-1}}  {\rm d}x\right)^{\frac{h^+-1}{h^+}}\left( \int_{\Omega}u_k^{(\theta_n+1)h^+}{\rm d}x\right)^{\frac{1}{h^+}}\nonumber\\
& \leq \vert \Omega\vert^{\frac{h^+-g^-}{h^+}}\Vert u\Vert_{h^+}^{g^--1} \Vert u_k\Vert_{\sigma_n}^{\frac{\sigma_n}{h^+}}\nonumber\\
& =C_8\Vert u_k\Vert_{\sigma_n}^{\frac{\sigma_n}{h^+}},\ \ \text{for all}\ n\in\mathbb{N},
\end{align}
 where $\displaystyle{C_8=\vert \Omega\vert^{\frac{h^+-g^-}{h^+}}\Vert u\Vert_{h^+}^{g^--1}}$.\\
Putting together $(\ref{588})$, $(\ref{7888})$, $(\ref{888})$ and $(\ref{455169965})$, we find that
\begin{align}\label{988}
\int_{\Omega} \left| \nabla u_k^{\frac{\theta_n +g^-}{g^-}}\right|^{g^-}{\rm d}x+\int_{\Omega} \left|  u_k^{\frac{\theta_n +g^-}{g^-}}\right|^{g^-}{\rm d}x & \leq \left( \theta_n+1\right)^{g^-}\left[ \left( \frac{C_6}{(\theta_n +1)G(1)}+C_8\right) \Vert u_k\Vert_{\sigma_n}^{\frac{\sigma_n }{h^+}}+C_7\Vert u_k\Vert_{\sigma_n}^{\frac{\sigma_n }{h^+}-1}\right]\nonumber\\
& \leq \left( \theta_n+1\right)^{g^-}\left[ \left( C_6+C_8\right) \Vert u_k\Vert_{\sigma_n}^{\frac{\sigma_n }{h^+}}+C_7\Vert u_k\Vert_{\sigma_n}^{\frac{\sigma_n }{h^+}-1}\right],\ \ (\text{since}\ (\theta_n +1)G(1)\geq 1)
\end{align}
for all $n\in\mathbb{N}$. 
Since $g^-<\widehat{g}=\frac{Ng^-}{N-g^-}=g^-_*$, then the embedding $W^{1,G}(\Omega)\hookrightarrow W^{1,g^-}(\Omega)\hookrightarrow L^{\widehat{g}}(\Omega)$ are continuous. Moreover,  there is $C_9>0$ such that
\begin{equation}\label{1088}
\left\lVert u_k^{\frac{\theta_n +g^-}{g^-}}\right\rVert_{\widehat{g}}^{g^-}\leq C_9\left\lVert  u_k^{\frac{\theta_n +g^-}{g^-}}\right\rVert_{W^{1,g^-}(\Omega)}^{g^-},\ \ \text{for all}\ n\in\mathbb{N}.
\end{equation}
From $(\ref{988})$ and $(\ref{1088})$, we obtain
\begin{align}\label{7777775}
\left\lVert u_k^{\frac{\theta_n +g^-}{g^-}}\right\rVert_{\widehat{g}}^{g^-}\leq C_9\left( \theta_n+1\right)^{g^-}\left[ \left( C_6+C_8\right) \Vert u_k\Vert_{\sigma_n}^{\frac{\sigma_n }{h^+}}+C_7\Vert u_k\Vert_{\sigma_n}^{\frac{\sigma_n }{h^+}-1}\right]
\end{align}
for all $n\in\mathbb{N}$. From the definition of the sequence $\lbrace\sigma_n\rbrace_{n\in\mathbb{N}}$, we have $\displaystyle{\frac{\sigma_{n+1}}{\widehat{g}}=\frac{\theta_{n}+g^-}{g^-}}$.\\
It follows, by $(\ref{7777775})$, that
\begin{equation}\label{118888}
\Vert u_k\Vert^{\sigma_{n+1}\frac{g^-}{\widehat{g}}}_{\sigma_{n+1}}\leq \left( \theta_n+1\right)^{g^-}C_9\left[ \left( C_6+C_8\right) \Vert u_k\Vert_{\sigma_n}^{\frac{\sigma_n }{h^+}}+C_7\Vert u_k\Vert_{\sigma_n}^{\frac{\sigma_n }{h^+}-1}\right]
\end{equation}
for all $n\in\mathbb{N}$.\\
Let $k\longrightarrow+\infty$ in $(\ref{118888})$, and using the monotone convergence theorem  , we get
\begin{equation}\label{1188}
\Vert u\Vert^{\sigma_{n+1}\frac{g^-}{\widehat{g}}}_{\sigma_{n+1}}\leq \left( \theta_n+1\right)^{g^-}C_9\left[ \left( C_6+C_8\right) \Vert u\Vert_{\sigma_n}^{\frac{\sigma_n }{h^+}}+C_7\Vert u\Vert_{\sigma_n}^{\frac{\sigma_n }{h^+}-1}\right]
\end{equation}
 We distinguish two cases.\\
\textbf{Case 1:} If $\left\lbrace n\in\mathbb{N},\ \Vert u\Vert_{\sigma_n}\leq 1\right\rbrace $ is unbounded. Then, without loss of generality, we may assume that
\begin{equation}\label{746654548}
\Vert u\Vert_{\sigma_n}\leq 1,\ \ \text{for all}\ n\in \mathbb{N}.
\end{equation}
Hence,
$$\Vert u\Vert_{\infty} \leq 1$$
since, $ \sigma_n\longrightarrow +\infty$ as $n\longrightarrow +\infty$ and $u\in L^s(\Omega)$ for all $s\geq 1$.
So, we are done with $M=1$.\\
\textbf{Case 2:} If $\left\lbrace n\in\mathbb{N},\ \Vert u\Vert_{\sigma_n}\leq 1\right\rbrace $ is bounded. Then, without loss of generality, we can suppose that
\begin{equation}\label{74655654548}
\Vert u\Vert_{\sigma_n}> 1,\ \ \text{for all}\ n\in \mathbb{N}.
\end{equation}
 From $(\ref{1188})$ and $(\ref{74655654548})$, we find that
\begin{equation}\label{555559}
\Vert u\Vert^{\sigma_{n+1}\frac{g^-}{\widehat{g}}}_{\sigma_{n+1}}\leq  C_{10}\Vert u\Vert_{\sigma_n}^{\frac{\sigma_n}{h^+}},\ \ \text{for all}\ n\in\mathbb{N}
\end{equation}
where $C_{10}(n)=\left( \theta_n+1\right)^{g^-} C_9 \left( C_6+C_7 +C_8\right)$.\\

We want to remark that
\begin{align}\label{8945666}
C_9\left( C_6+C_7+C_8\right) & =C_9\left[ c_2\left((1+h(1))\vert \Omega\vert^{1-\frac{1}{h^+}}+ h(1) \Vert u\Vert_{h^+}^{h^+-1}\right)+\vert \Omega\vert^{\frac{h^+-g^-}{h^+}}\Vert u\Vert_{h^+}^{g^--1}+\vert\Omega\vert^{1-\frac{1}{h^+}+\frac{1}{\sigma_n}}\right]\nonumber\\
& \leq  C_9\left[ c_2\left((1+h(1))\vert \Omega\vert^{1-\frac{1}{h^+}}+  h(1)\Vert u\Vert_{h^+}^{h^+-1}\right)+\vert \Omega\vert^{\frac{h^+-g^-}{h^+}}\Vert u\Vert_{h^+}^{g^--1}+\vert\Omega\vert+1\right]\nonumber\\
& =C_{11},\ \ \text{for all}\ n\in \mathbb{N}.
\end{align}
Hence $C_{11}> 0$ is independent of $n$. Moreover, we have that
\begin{equation}\label{9564872}
(\theta_n+1)^{g^-}=(\frac{\sigma_n}{h^+})^{g^-}\leq (\sigma_n)^{g^-}\leq (\sigma_{n+1})^{g^-},\ \ \text{for all}\ n\in\mathbb{N}.
\end{equation}
From $(\ref{555559})$, $(\ref{8945666})$ and $(\ref{9564872})$, we obtain
\begin{equation}\label{14156699}
\Vert u\Vert^{\sigma_{n+1}\frac{g^-}{\widehat{g}}}_{\sigma_{n+1}}\leq (\sigma_{n+1})^{g^-} C_{11}\Vert u\Vert_{\sigma_n}^{\frac{\sigma_n}{h^+}},\ \ \text{for all}\ n\in\mathbb{N}.
\end{equation}
Therefore, from \cite[Theorem 6.2.6, p. 737]{npl}, we find that
\begin{equation}  \label{688888}
\Vert u\Vert_{\sigma_{n+1}}\leq M,\ \ \text{for all}\ n\in \mathbb{N}
\end{equation}
for some $M\left( c_2, h(1), g^-, \vert \Omega\vert, \Vert u \Vert_{h^+}\right)\geq 0$.\\
On the other hand, by the hypotheses of the proposition, we have that
\begin{equation}  \label{688}
u\in L^s(\Omega),\ \ \text{for all}\ 1\leq s<\infty.
\end{equation}
Exploiting $(\ref{688888})$,$(\ref{688})$ and the fact that $\sigma_n\longrightarrow +\infty$ as $n\longrightarrow +\infty$, we deduce that
$$\Vert u\Vert_\infty\leq M.$$
This en{\rm d}s the proof.
\end{proof}
\begin{proof}[\textbf{Proof of Theorem \ref{thmC11}}]
Let
$$
\begin{array}{ll}
\mathcal{A}(x,\eta)=a(\vert \eta \vert)\eta, & \text{for all}\ x\in \Omega\ \text{and}\ \eta \in \mathbb{R}^N\\
\ & \ \\
\mathcal{B}(x,t)=f(x,t),& \text{for all}\ x\in \Omega\ \text{and}\  t\in \mathbb{R}\\
\ & \ \\
\psi(x,t)=b(x)\vert t\vert^{p-2}t,& \text{for all}\ x\in \partial\Omega\ \text{and}\  t\in \mathbb{R}
\end{array}
$$
in problem $(\ref{A})$.
Then, $\mathcal{A},\mathcal{B}$ and $\psi$ satisfy the growth conditions $(\ref{7546999882})-(\ref{7546999882222})$ and the problem $(\ref{A})$ turns to  $(\ref{P})$.
By the Propositions \ref{prop1} and \ref{prop2}, we conclude that every weak solution $u\in W^{1,G}(\Omega)$ of problem $(\ref{P})$ belongs to $L^{\infty}(\Omega)$ and $\Vert u\Vert_{\infty}\leq M=M(c_2=\Vert \widehat{a}\Vert_{\infty}, h(1), g^-,\vert \Omega\vert, \Vert u \Vert_{h^+})$. This en{\rm d}s the proof.
\end{proof}
\section{$W^{1,G}(\Omega)$ versus $C^1(\overline{\Omega})$ local minimizers}

In this section, using the regularity theory of Lieberman \cite{2}, we extend the result of Brezis and Nirenberg’s  \cite{29} to the problem $(\ref{P})$.

\begin{prop}\label{prop3}
let $u_0\in W^{1,G}(\Omega)$ be a local $C^1(\overline{\Omega})$-minimizer of $J$ (see Definition \ref{def333}), then $u_0$ is a weak solution for problem  $(\ref{P})$ and $u_0\in C^{1,\alpha}(\overline{\Omega})$, for some $\alpha\in (0,1)$.
\end{prop}
\begin{proof}[Proof]
By hypothesis $u_0$ is a local $C^1(\overline{\Omega})$-minimizer of $J$, for every $v\in C^1(\overline{\Omega})$ and  $t>0$ small enough, we have $J(u_0)\leq J(u_0+tv)$. Hence,
\begin{equation}\label{20}
0\leq \langle J^{'}(u_0),v\rangle\ \ \text{for all}\ v\in C^1(\overline{\Omega}).
\end{equation}
Since $C^1(\overline{\Omega})$ is dense in $W^{1,G}(\Omega)$, from $(\ref{20})$ we infer that $J^{'}(u_0)=0$. Namely,
\begin{align}\label{21}
\int_{\Omega} a(\vert\nabla u_0\vert)\nabla u_0.\nabla v {\rm d}x +\int_{\partial\Omega}b(x)\vert u_0\vert^{p-2}u_0 v {\rm d}\gamma
 & =\int_{\Omega}f(x,u_0)v{\rm d}x,\ \ \text{for all}\ v\in W^{1,G}(\Omega).
\end{align}

By the nonlinear Green's identity, we get
\begin{align}\label{22}
\int_{\Omega} a(\vert\nabla u_0\vert)\nabla u_0.\nabla v {\rm d}x=-\int_{\Omega} \text{div}(a(\vert\nabla u_0\vert)\nabla u_0). v {\rm d}x+ \int_{\partial\Omega}a(\vert \nabla u_0\vert)\nabla u_0. \nu \ v {\rm d}\gamma,
\end{align}
$\text{for all}\ v\in W^{1,G}(\Omega)$.
It follows that,
\begin{align}\label{23}
\int_{\Omega} a(\vert\nabla u_0\vert)\nabla u_0.\nabla v {\rm d}x=-\int_{\Omega} \text{div}(a(\vert\nabla u_0\vert)\nabla u_0). v {\rm d}x,\ \ \text{for all}\ v\in W^{1,G}_0(\Omega).
\end{align}
Hence, by $(\ref{21})$
\begin{align*}
 -\int_{\Omega} \text{div}(a(\vert\nabla u_0\vert)\nabla u_0). v {\rm d}x  =\int_{\Omega}f(x,u_0)v{\rm d}x,\ \ \text{for all}\ v\in W^{1,G}_0(\Omega),
\end{align*}
which gives,
\begin{align}\label{24}
& -\text{div}(a(\vert\nabla u_0(x)\vert)\nabla u_0(x))
 =f(x,u_0(x)),\ \ \text{for almost}\ x\in \Omega.
\end{align}
From $(\ref{21})$, $(\ref{22})$ and $(\ref{24})$, we obtain
\begin{equation}\label{25}
\left\langle a(\vert \nabla u_0\vert)\frac{\partial u_0}{\partial \nu}+b(x)\vert u_0\vert^{p-2}u_0 ,v\right\rangle_{\partial\Omega}=0\ \ \text{for all}\ v\in W^{1,G}(\Omega).
\end{equation}
It follows that
$$a(\vert \nabla u_0\vert)\frac{\partial u_0}{\partial \nu}+b(x)\vert u_0\vert^{p-2}u_0=0\ \ \text{on}\ \partial\Omega.$$
So, $u_0\in W^{1,G}(\Omega)$ is a weak solution for the  problem $(\ref{P})$.
From Theorem \ref{thmC11}, we have that $u_0\in L^{\infty}(\Omega)$.\\

We define $A:\overline{\Omega}\times\mathbb{R}^N\rightarrow\mathbb{R}^N$, $B:\overline{\Omega}\times\mathbb{R}\rightarrow\mathbb{R}$ and $\phi:\partial\Omega\times\mathbb{R}\rightarrow\mathbb{R}$ by
\begin{equation}\label{789}
\left\lbrace
\begin{array}{ll}
A(x,\eta)=a(\vert \eta\vert)\eta; \\
\ \\
B(x,t)=f(x,t);\\
\ \\
\phi(x,t)=b(x)\vert t\vert^{p-2}t.
\end{array}
\right.
\end{equation}
It is easy to show that, for $x,y\in\overline{\Omega},\  \eta \in \mathbb{R}^N\setminus\lbrace 0\rbrace,\ \xi\in\mathbb{R}^N,\ t\in \mathbb{R},$ the following estimations hold:
\begin{equation}\label{159}
 A(x,0)=0,
\end{equation}
\begin{equation}\label{1599}
 \sum_{i,j=1}^N \frac{\partial (A)_j}{\partial\eta_i}(x,\eta)\xi_i\xi_j \geq \frac{g(\vert \eta \vert)}{\vert \eta \vert}\vert\xi\vert^2,
\end{equation}
\begin{equation}\label{15999}
\sum_{i,j=1}^N \left|\frac{\partial (A)_j}{\partial\eta_i}(x,\eta) \right|\vert \eta\vert\leq c(1+g(\vert \eta\vert)),
\end{equation}
\begin{equation}\label{159999}
\vert A(x,\eta)-A(y,\eta)\vert\leq c(1+g(\vert \eta\vert))(\vert x-y\vert^\theta),\ \text{for some}\ \theta \in (0,1),
\end{equation}
\begin{equation}\label{1599999}
\vert B(x,t)\vert \leq c\left( 1+h(\vert t\vert)\right).
\end{equation}
Indeed: inequalities $(\ref{159})$ , $(\ref{159999})$  and $(\ref{1599999})$ are evident.\\
For $x\in\overline{\Omega},\  \eta \in \mathbb{R}^N\backslash\lbrace 0\rbrace,\ \xi\in\mathbb{R}^N$, we have
\begin{equation}\label{12254996}
D_\eta(A(x,\eta))\xi=a(\vert \eta\vert)\xi +a^{'}(\vert \eta\vert)\frac{\langle\eta ,\xi\rangle_{\mathbb{R}^N}}{\vert  \eta\vert}\eta
\end{equation}
and
\begin{equation}\label{122484996}
\langle D_\eta(A(x,\eta))\xi,\xi\rangle_{\mathbb{R}^N}=a(\vert \eta\vert)\langle\xi,\xi\rangle_{\mathbb{R}^N} +a^{'}(\vert \eta\vert)\frac{\left[ \langle\eta ,\xi\rangle_{\mathbb{R}^N}\right]^2 }{\vert  \eta\vert}
\end{equation}
where $\langle,\rangle_{\mathbb{R}^N}$ is the inner product in $\mathbb{R}^N$. Hence, we have
 the following derivative
\begin{align}\label{666}
D_\eta (a(\vert \eta\vert)\eta)=\frac{a^{'}(\vert \eta\vert)}{\vert \eta\vert}\eta\eta^{T}+a(\vert \eta\vert)I_N & =a(\vert \eta\vert)\left( I_N+\frac{a^{'}(\vert \eta\vert)\vert \eta\vert }{a(\vert\eta\vert)}\frac{1}{\vert \eta\vert^2}M_N(\eta,\eta)\right)
\end{align}
for all $\eta\in \mathbb{R}^N\backslash\lbrace 0\rbrace$, where $\eta^T$ is the transpose of $\eta$, $I_N$ is the unit matrix in $M_N(\mathbb{R})$ and
\begin{equation}\label{455666699}
M_N(\eta,\eta)=\eta \eta^{T}=\begin{pmatrix}
    \eta_1^2 & \eta_1\eta_2 &\cdots & \eta_1\eta_N \\
    \eta_2\eta_1 & \eta_2^2 &\cdots & \eta_2\eta_N\\
    \vdots & \vdots& \ddots& \vdots\\
    \eta_N\eta_1& \eta_N\eta_2 & \cdots & \eta_N^2\\
    \end{pmatrix}
    \quad
\end{equation}
for all $\eta\in\mathbb{R}^N$.\\
Note that, for all $\eta\in \mathbb{R}^N$, we have
\begin{equation}\label{4546789123}
\lVert M_N(\eta,\eta)\rVert_{\mathbb{R}^N}=\sum_{i,j=1}^{N}\vert \eta_i \eta_j\vert=\left( \sum_{i=1}^{N}\vert \eta_i\vert\right) ^{2}\leq N\sum_{i=1}^{N}\vert \eta_i\vert^2=N\vert \eta\vert^2
\end{equation}
where $\Vert .\Vert_{\mathbb{R}^N}$ is a norm on $M_N(\mathbb{R})$.\\
From $(\ref{122484996})$ and assumption $(g_3)$, we have
\begin{align}\label{954621866}
\sum_{i,j=1}^N \frac{\partial (A)_j}{\partial\eta_i}(x,\eta)\xi_i\xi_j & =\langle D_\eta(A(x,\eta))\xi,\xi\rangle\nonumber\\
& =a(\vert \eta\vert)\langle\xi,\xi\rangle_{\mathbb{R}^N} +a^{'}(\vert \eta\vert)\frac{\left[ \langle\eta ,\xi\rangle_{\mathbb{R}^N}\right]^2 }{\vert  \eta\vert}\nonumber\\
& =a(\vert \eta\vert)\left[ \langle\xi,\xi\rangle_{\mathbb{R}^N} +\frac{a^{'}(\vert \eta\vert)\vert  \eta\vert}{a(\vert \eta\vert)}\frac{\left[ \langle\eta,\xi\rangle_{\mathbb{R}^N}\right] ^2}{\vert  \eta\vert^2}\right] \nonumber\\
& \geq \frac{g(\vert \eta\vert)}{\vert \eta \vert}\vert \xi\vert^2
\end{align}
for all $x\in\overline{\Omega},\  \eta \in \mathbb{R}^N\backslash\lbrace 0\rbrace,\ \xi\in\mathbb{R}^N$.\\
Moreover, from  $(\ref{666})$, $(\ref{4546789123})$  and assumption $(g_3)$, we find that
\begin{align}\label{455645489}
\sum_{i,j=1}^N \left|\frac{\partial (A)_j}{\partial\eta_i}(x,\eta) \right|\vert \eta\vert &=\lVert D_\eta(A(x,\eta))\rVert_{\mathbb{R}^N}\vert \eta\vert\nonumber\\
&\leq \left( \lVert I_N\rVert_{\mathbb{R}^N}+\frac{a^{'}(\vert \eta\vert)\vert \eta\vert }{a(\vert\eta\vert)}\frac{1}{\vert \eta\vert^2}\lVert M_N(\eta,\eta)\rVert_{\mathbb{R}^N}\right)g(\vert\eta \vert)\nonumber\\
&\leq \left( 1+\frac{a^{'}(\vert \eta\vert)\vert \eta\vert }{a(\vert\eta\vert)}\right)Ng(\vert\eta \vert)\nonumber\\
&\leq a^+Ng(\vert\eta \vert)\nonumber\\
&\leq a^+N(1+g(\vert\eta \vert))
\end{align}
for all $x\in\overline{\Omega},\  \eta \in \mathbb{R}^N\backslash\lbrace 0\rbrace$.\\
 The non-linear regularity result of Lieberman \cite[p. 320]{2} implies the existence of $\alpha\in(0,1)$ and $M_0\geq 0$ such that
$$u_0\in C^{1,\alpha}(\overline{\Omega})\ \ \text{and}\ \ \Vert u_0\Vert_{C^{1,\alpha}(\overline{\Omega})}\leq M_0.$$
This en{\rm d}s the proof.
\end{proof}
\begin{prop}\label{prop4}
Under the assumptions $(G)$ and $(H)$, if $u_0\in W^{1,G}(\Omega)$ is a local $C^1(\overline{\Omega})$-minimizer of $J$ (see Definition \ref{def333}), then $u_0\in W^{1,G}(\Omega)$ is also a local $W^{1,G}(\Omega)$-minimizer of $J$ (see Definition \ref{def333}).
\end{prop}
\begin{proof}[Proof]Let $u_0$ be a local $C^1(\overline{\Omega})$-minimizer of $J$, then, by Proposition $\ref{prop3}$, we have
\begin{equation}
u_0\in L^{\infty}(\Omega)\ \text{and}\ u_0\in C^{1,\alpha}(\overline{\Omega})\ \ \text{for some}\ \alpha\in(0,1).
\end{equation}
To prove that $u_0$ is a local $W^{1,G}(\Omega)$-minimizer of $J$, we argue by contradiction. Suppose that $u_0$ is not a local $W^{1,G}(\Omega)$-minimizer of $J$. Let $\varepsilon\in(0,1)$ and define
$$B(u_0,\varepsilon)=\left\lbrace v\in W^{1,G}(\Omega):\ \mathcal{K}(v-u_0)\leq \varepsilon\right\rbrace, $$
recall that $\displaystyle{\mathcal{K}(v-u_0)=\int_{\Omega}G(\vert \nabla(v-u_0)\vert){\rm d}x+\int_{\Omega}G(\vert v-u_0\vert){\rm d}x}$.\\ We consider the following minimization problem:
\begin{equation}\label{26}
m_\varepsilon =\inf \left\lbrace J(v):\ v\in B(u_0,\varepsilon)\right\rbrace.
\end{equation}
By the hypothesis of contradiction and assumption $(H)$, we have
\begin{equation}\label{27}
-\infty<m_\varepsilon <J(u_0).
\end{equation}
The set $B(u_0,\varepsilon)$ is bounded, closed and convex subset of $W^{1,G}(\Omega)$ and is a neighbourhood of $u_0\in W^{1,G}(\Omega)$. Since $f(x,t)$ satisfies the assumption $(H)$, the functional $J:W^{1,G}(\Omega)\rightarrow\mathbb{R}$ is weakly lower semicontinuous. So, From the Weierstrass theorem  there exist $v_\varepsilon\in B(u_0,\varepsilon)$ such that $m_\varepsilon=J(v_\varepsilon)$. Moreover, by $(\ref{27})$, we deduce that $v_\varepsilon\neq 0$.\\
Now, using the Lagrange multiplier rule \cite[p. 35]{13}, we can find $\lambda_\varepsilon\geq 0$ such that
$$\langle J^{'}(v_\varepsilon),v\rangle+\lambda_\varepsilon\langle \mathcal{K}^{'}(v_\varepsilon-u_0),v\rangle=0\ \ \text{for all}\ v\in W^{1,G}(\Omega),$$
which implies
\begin{align}\label{29}
\langle J^{'}(v_\varepsilon),v\rangle+\lambda_\varepsilon\langle \mathcal{K}^{'}(v_\varepsilon-u_0),v\rangle & =\int_{\Omega} a(\vert\nabla v_\varepsilon\vert)\nabla v_\varepsilon.\nabla v {\rm d}x +\int_{\partial\Omega}b(x)\vert v_\varepsilon\vert^{p-2}v_\varepsilon v {\rm d}\gamma\nonumber\\
 &+\lambda_\varepsilon\int_{\Omega} a(\vert\nabla( v_\varepsilon-u_0)\vert)\nabla (v_\varepsilon-u_0).\nabla v {\rm d}x -\int_{\Omega}f(x, v_\varepsilon)v{\rm d}x\nonumber\\
& +\lambda_\varepsilon\int_{\Omega} a(\vert v_\varepsilon-u_0\vert) (v_\varepsilon-u_0) v {\rm d}x\nonumber\\
& =0
\end{align}
for all $v\in W^{1,G}(\Omega)$.\\
In the other side, from Proposition \ref{prop3}, we see that $u_0\in W^{1,G}(\Omega)$ is a weak solution for the problem $(\ref{P})$. Hence,
\begin{align}\label{78956241}
\int_{\Omega} a(\vert\nabla u_0\vert)\nabla u_0.\nabla v {\rm d}x +\int_{\partial\Omega}b(x)\vert u_0\vert^{p-2} u_0 v {\rm d}\gamma-\int_{\Omega}f(x, u_0)v{\rm d}x=0
\end{align}
for all $v\in W^{1,G}(\Omega)$.\\

Next, we have to show that $v_\varepsilon$ belongs to $L^{\infty}(\Omega)$ and hence to $C^{1,\alpha}(\overline{\Omega})$. We distinguish three cases.\\
\textbf{Case 1:} If $\lambda_\varepsilon=0$ with $\varepsilon\in (0,1]$, we find that $v_\varepsilon$ solves the Robin boundary value problem $(\ref{P})$. As in Proposition $\ref{prop3}$, we prove that $v_\varepsilon\in C^{1,\alpha}(\overline{\Omega})$ for some $\alpha\in (0,1)$ and there is $M_1\geq 0$ (independent of $\varepsilon$) such that
$$\Vert v_\varepsilon\Vert_{C^{1,\alpha}(\overline{\Omega})}\leq M_1.$$
\textbf{Case 2:} If $0<\lambda_\varepsilon\leq 1$ with $\varepsilon\in (0,1]$.
Multiplying $(\ref{78956241})$ by $\lambda_\varepsilon>0$ and adding $(\ref{29})$, we get
\begin{align}\label{7458869}
\int_{\Omega} a(\vert\nabla v_\varepsilon\vert)\nabla v_\varepsilon.\nabla v {\rm d}x &+\lambda_\varepsilon\int_{\Omega} a(\vert\nabla u_0\vert)\nabla u_0.\nabla v {\rm d}x
 +\lambda_\varepsilon\int_{\Omega} a(\vert\nabla( v_\varepsilon-u_0)\vert)\nabla (v_\varepsilon-u_0).\nabla v {\rm d}x \nonumber\\
 & +\lambda_\varepsilon\int_{\partial\Omega}b(x)\vert u_0\vert^{p-2} u_0 v {\rm d}\gamma +\int_{\partial\Omega}b(x)\vert v_\varepsilon\vert^{p-2}v_\varepsilon v {\rm d}\gamma\nonumber\\
  &=\lambda_\varepsilon\int_{\Omega}f(x, u_0)v{\rm d}x +\int_{\Omega}f(x, v_\varepsilon)v{\rm d}x -\lambda_\varepsilon\int_{\Omega} a(\vert v_\varepsilon-u_0\vert) (v_\varepsilon-u_0) v {\rm d}x
\end{align}
for all $v\in W^{1,G}(\Omega)$.\\

Let $A_\varepsilon:\overline{\Omega}\times\mathbb{R}^N\rightarrow\mathbb{R}^N$, $B_\varepsilon:\overline{\Omega}\times\mathbb{R}\rightarrow\mathbb{R}$ and $\phi_\varepsilon:\partial\Omega\times\mathbb{R}\rightarrow\mathbb{R}$ defined by
\begin{equation}\label{788}
\left\lbrace
\begin{array}{ll}
A_\varepsilon(x,\eta)=a(\vert \eta \vert)\eta+\lambda_\varepsilon a(\vert \eta -\nabla u_0\vert)(\eta-\nabla u_0)+\lambda_\varepsilon a(\vert \nabla u_0\vert)\nabla u_0;\\
\ \\
B_\varepsilon(x,t)=f(x,t)+\lambda_\varepsilon f(x,u_0)- \lambda_\varepsilon a(\vert t-u_0\vert)(t-u_0);\\
\ \\
\phi_\varepsilon(x,t)=b(x)\left( \vert t\vert^{p-2}t+\lambda_\varepsilon \vert u_0\vert^{p-2}u_0 \right).
\end{array}
\right.
\end{equation}
It is clear that $A_\varepsilon\in C(\overline{\Omega}\times\mathbb{R}^N,\mathbb{R}^N)$.  Hence, the equation $(\ref{7458869})$  is the weak formulation of the following Robin boundary value problem
$$
\left\lbrace
\begin{array}{ll}
-\text{div}(A_\varepsilon(x,\nabla v_\varepsilon)) =B_\varepsilon(x,v_\varepsilon)\ & \text{on}\  \Omega,\\
\ & \ \\
A_\varepsilon(x,\nabla v_\varepsilon).\nu + \phi_\varepsilon(x,v_\varepsilon)=0 & \text{on} \ \partial\Omega,
\end{array}
\right.
$$
where $\nu$ is the inner normal to $\partial\Omega$.\\
 From Lemma \ref{lem444} and assumption $(G)$, for $\eta\in \mathbb{R}^n$ and $x\in\Omega$ , we have
\begin{align}\label{45779756}
\langle A_\varepsilon(x,\eta),\eta\rangle_{\mathbb{R}^N}& =\langle a(\vert \eta\vert)\eta,\eta\rangle_{\mathbb{R}^N}+ \lambda_\varepsilon\langle a(\vert \eta-\nabla u_0\vert)(\eta-\nabla u_0),\eta -\nabla u_0-(-\nabla u_0)\rangle_{\mathbb{R}^N}\nonumber\\
& -\lambda_\varepsilon\langle a(\vert -\nabla u_0\vert)(-\nabla u_0),\eta -\nabla u_0-(-\nabla u_0)\rangle_{\mathbb{R}^N}\nonumber\\
& \geq g(\vert \eta\vert)\vert \eta\vert\nonumber\\
&\geq G(\vert \eta\vert)
\end{align}
and
\begin{align}\label{457779756}
\vert A_\varepsilon(x,\eta)\vert& \leq a(\vert \eta\vert)\vert\eta\vert+ \lambda_\varepsilon a(\vert \eta-\nabla u_0\vert)\vert\eta-\nabla u_0\vert+\lambda_\varepsilon a(\vert \nabla u_0\vert)\vert \nabla u_0\vert\nonumber\\
& \leq g(\vert \eta\vert)+g(\vert \eta-\nabla u_0\vert)+g(\vert \nabla u_0\vert)\ \ (\text{since}\ 0<\lambda_\varepsilon\leq 1)\nonumber\\
&\leq g(\vert \eta\vert)+g(\vert \eta\vert+\vert\nabla u_0\vert)+g(\vert \nabla u_0\vert)\nonumber\\
& \leq c_0g(\vert \eta\vert)+c_1\ (\text{using Lemma \ref{lem1} and the monotonicity of}\ g).
\end{align}
 Then, $A_\varepsilon, B_\varepsilon$ and $\phi_\varepsilon$ satisfy the corresponding growth conditions $(\ref{7546999882})-(\ref{7546999882222})$.
So, using the Propositions \ref{prop1} and \ref{prop2},  we obtain that $v_\varepsilon\in L^{\infty}(\Omega)$.\\
 It remains, using the regularity theorem of Lieberman, to show that $v_\varepsilon\in C^{1,\alpha}(\overline{\Omega})$ for some $\alpha\in (0,1)$. So, we need to prove that $A_\varepsilon$ and $B_\varepsilon$ satisfy the corresponding  $(\ref{159})-(\ref{1599999})$.
The inequalities $(\ref{159})$ and $(\ref{1599999})$ are evident. The inequality $(\ref{159999})$ follows from Lemma \ref{lem3} and the fact that $\nabla u_0$ is H\"older continuous.\\

As in $(\ref{12254996})$ and $(\ref{122484996})$, we have
\begin{equation}\label{1225445996}
D_\eta(a(\vert\eta-\nabla u_0\vert)(\eta-\nabla u_0))\xi=a(\vert \eta-\nabla u_0\vert)\xi +a^{'}(\vert \eta-\nabla u_0\vert)\frac{\langle \eta-\nabla u_0 ,\xi\rangle_{\mathbb{R}^N}}{\vert \eta-\nabla u_0\vert}(\eta-\nabla u_0)
\end{equation}
and
\begin{equation}\label{12254445996}
\langle D_\eta(a(\vert\eta-\nabla u_0\vert)(\eta-\nabla u_0))\xi,\xi\rangle_{\mathbb{R}^N}=a(\vert \eta-\nabla u_0\vert)\langle\xi,\xi\rangle_{\mathbb{R}^N} +a^{'}(\vert \eta-\nabla u_0\vert)\frac{\left[ \langle \eta-\nabla u_0 ,\xi\rangle_{\mathbb{R}^N}\right] ^2}{\vert \eta-\nabla u_0\vert}
\end{equation}
for all $x\in\overline{\Omega},\  \eta \in \mathbb{R}^N\backslash\lbrace \nabla u_0\rbrace,\ \xi\in\mathbb{R}^N$.\\
Exploiting $(\ref{954621866})$, $(\ref{12254445996})$ and assumption $(g_3)$, we infer that
\begin{align}\label{54548998}
\sum_{i,j=1}^N \frac{\partial (A_\varepsilon)_j}{\partial\eta_i}(x,\eta)\xi_i\xi_j
&=\langle D_\eta (A)(x,\eta)\xi,\xi\rangle_{\mathbb{R}^N}\nonumber\\
&+\lambda_\varepsilon a(\vert \eta-\nabla u_0\vert )\left( \langle\xi,\xi\rangle_{\mathbb{R}^N}+\frac{a^{'}(\vert \eta-\nabla u_0\vert )\vert\eta-\nabla u_0\vert}{a(\vert\eta-\nabla u_0\vert)}\frac{\left[ \langle\eta-\nabla u_0,\xi\rangle_{\mathbb{R}^N}\right] ^2}{\vert\eta-\nabla u_0\vert^2}\right)\nonumber\\
& \geq  \langle D_\eta (A)(x,\eta)\xi,\xi\rangle_{\mathbb{R}^N}\nonumber\\
& \geq \frac{g(\vert \eta\vert)}{\vert \eta\vert}\vert\xi\vert^2
\end{align}
for all $x\in\overline{\Omega},\  \eta \in \mathbb{R}^N\backslash\lbrace \nabla u_0\rbrace,\ \xi\in\mathbb{R}^N$.\\
Note that the derivative of $A_\varepsilon$ has the form
\begin{equation}\label{7845661236}
D_\eta (A_\varepsilon(x,\eta))=D_\eta (A(x,\eta))+\lambda_\varepsilon a(\vert \eta-\nabla u_0\vert)\left( I_N+\frac{a^{'}(\vert \eta-\nabla u_0\vert)\vert \eta-\nabla u_0\vert }{a(\vert\eta-\nabla u_0\vert)}\frac{1}{\vert \eta-\nabla u_0\vert^2}M_N(\eta-\nabla u_0,\eta-\nabla u_0)\right)
\end{equation}
for all $x\in\overline{\Omega},\  \eta \in \mathbb{R}^N\backslash\lbrace \nabla u_0\rbrace$, where $M_N(\eta-\nabla u_0,\eta-\nabla u_0)$ is defined in $(\ref{455666699})$.\\
As in $(\ref{4546789123})$, we have
\begin{equation}\label{4545556789123}
\lVert M_N(\eta-\nabla u_0,\eta-\nabla u_0)\rVert_{\mathbb{R}^N}\leq N\vert \eta-\nabla u_0\vert^2.
\end{equation}
In light of  $(\ref{455645489})$, $(\ref{7845661236})$, $(\ref{4545556789123})$ and assumption $(g_3)$, we see that
 \begin{align}\label{7458896323}
 \sum_{i,j=1}^N \left|\frac{\partial (A_\varepsilon)_j}{\partial\eta_i}(x,\eta) \right|\vert \eta\vert& =\lVert D_\eta (A_\varepsilon(x,\eta))\rVert_{\mathbb{R}^N}\vert \eta \vert\nonumber\\
 & \leq a^+N a(\vert\eta \vert)\vert\eta \vert+\lambda_\varepsilon a(\vert \eta-\nabla u_0\vert)\vert \eta\vert \lVert I_N\rVert_{\mathbb{R}^N}\nonumber\\
 & +\lambda_\varepsilon a(\vert \eta-\nabla u_0\vert) \vert\eta\vert\left( \frac{a^{'}(\vert \eta-\nabla u_0\vert)\vert \eta-\nabla u_0\vert }{a(\vert\eta-\nabla u_0\vert)}\frac{\lVert M_N(\eta-\nabla u_0,\eta-\nabla u_0)\rVert_{\mathbb{R}^N}}{\vert \eta-\nabla u_0\vert^2}\right)\nonumber\\
 &\leq a^+N a(\vert\eta \vert)\vert\eta \vert+ \lambda_\varepsilon a^+N a(\vert \eta-\nabla u_0\vert) \vert\eta\vert\nonumber\\
 &\leq a^+N \vert\eta \vert\left( a(\vert\eta \vert)+  a(\vert \eta-\nabla u_0\vert) \right) \nonumber\\
 &\leq c(1+g(\vert \eta\vert))
 \end{align}
for all $x\in\overline{\Omega},\  \eta \in \mathbb{R}^N\backslash\lbrace \nabla u_0\rbrace$.\\
 So, from the regularity theorem of Lieberman \cite[p. 320]{2}, we can find $\alpha\in (0,1)$ and $M_2>0$, both independent from $\varepsilon$, such that
\begin{equation}\label{32}
 v_\varepsilon\in C^{1,\alpha}(\overline{\Omega}),\ \ \Vert v_\varepsilon\Vert_{C^{1,\alpha}(\overline{\Omega})}\leq M_2\ \ \text{for all}\ \varepsilon\in(0,1].
\end{equation}
\textbf{Case 3:} If $1<\lambda_\varepsilon$ with $\varepsilon\in (0,1]$. Multiplying $(\ref{78956241})$ with $-1$, setting $y_\varepsilon=v_\varepsilon-u_0$ in   $(\ref{29})$ and adding, we get
\begin{align}\label{745886669}
\int_{\Omega} a(\vert\nabla( y_\varepsilon+u_0)\vert)\nabla (y_\varepsilon+u_0).\nabla v {\rm d}x &-\int_{\Omega} a(\vert\nabla u_0\vert)\nabla u_0.\nabla v {\rm d}x
 +\lambda_\varepsilon\int_{\Omega} a(\vert\nabla y_\varepsilon\vert)\nabla y_\varepsilon.\nabla v {\rm d}x \nonumber\\
 & -\int_{\partial\Omega}b(x)\vert u_0\vert^{p-2} u_0 v {\rm d}\gamma +\int_{\partial\Omega}b(x)\vert y_\varepsilon+u_0\vert^{p-2}( y_\varepsilon+u_0) v {\rm d}\gamma\nonumber\\
  &=\int_{\Omega}f(x, y_\varepsilon+u_0)v{\rm d}x -\int_{\Omega}f(x, u_0)v{\rm d}x  -\lambda_\varepsilon\int_{\Omega} a(\vert y_\varepsilon\vert) y_\varepsilon v {\rm d}x
\end{align}
for all $v\in W^{1,G}(\Omega)$.\\
Defining again $\tilde{A}_\varepsilon:\overline{\Omega}\times\mathbb{R}^N\rightarrow\mathbb{R}^N$, $\tilde{B}_\varepsilon:\overline{\Omega}\times\mathbb{R}\rightarrow\mathbb{R}$ and $\tilde{\phi}_\varepsilon:\partial\Omega\times\mathbb{R}\rightarrow\mathbb{R}$ by
\begin{equation}\label{789968}
\left\lbrace
\begin{array}{ll}
\tilde{A}_\varepsilon(x,\eta)=a(\vert \eta \vert)\eta+\frac{1}{\lambda_\varepsilon} a(\vert \eta +\nabla u_0\vert)(\eta+\nabla u_0)-\frac{1}{\lambda_\varepsilon} a(\vert \nabla u_0\vert)\nabla u_0; \\
\ \\
\tilde{B}_\varepsilon(x,t)=\frac{1}{\lambda_\varepsilon}\left[ f(x,t+u_0)- f(x,u_0)\right] - a(\vert t\vert)t;\\
\ \\
\tilde{\phi}_\varepsilon(x,t)=\frac{1}{\lambda_\varepsilon}b(x)\left( \vert t+u_0\vert^{p-2}(t+u_0)-\vert u_0\vert^{p-2}u_0\right).
\end{array}
\right.
\end{equation}
It is clear that $A_\varepsilon\in C(\overline{\Omega}\times\mathbb{R}^N,\mathbb{R}^N)$. Rewriting $(\ref{745886669})$, we find the following equation
$$
\left\lbrace
\begin{array}{ll}
-\text{div}(\tilde{A}_\varepsilon(x,\nabla y_\varepsilon)) =\tilde{B}_\varepsilon(x,y_\varepsilon)\ & \text{on}\ \Omega,\\
\ & \ \\
\tilde{A}_\varepsilon(x,\nabla y_\varepsilon).\nu +\tilde{\phi}_\varepsilon(x,y_\varepsilon)=0 & \text{on} \ \partial\Omega,
\end{array}
\right.
$$
where $\nu$ is the inner normal to $\partial\Omega$.\\
Again, from Propositions $\ref{prop1}$ and $\ref{prop2}$, we conclude that $y_\varepsilon\in L^{\infty}(\Omega)$. By the same arguments used in case 2, we prove that $\tilde{A}_\varepsilon$ and $\tilde{B}_\varepsilon$ satisfy the corresponding inequalities $(\ref{159})-(\ref{1599999})$. So, the regularity theorem of Lieberman \cite[p. 320]{2}  implies the existence  of $\alpha\in (0,1)$ and $M_3\geq 0$ both independent of $\varepsilon$ such that
$$y_\varepsilon\in C^{1,\alpha}(\overline{\Omega}),\ \ \text{and}\ \Vert y_\varepsilon\Vert_{C^{1,\alpha}(\overline{\Omega})}\leq M_3.$$
Since $y_\varepsilon=v_\varepsilon-u_0$ and $u_0\in C^{1,\alpha}(\overline{\Omega})$, we infer that

$$v_\varepsilon\in C^{1,\alpha}(\overline{\Omega}),\ \ \text{and}\ \Vert v_\varepsilon\Vert_{C^{1,\alpha}(\overline{\Omega})}\leq M_3.$$

Let $\displaystyle{\varepsilon_n\searrow 0}$ as $n\longrightarrow +\infty$. Therefore, in the three cases, we have the same uniform $C^{1,\alpha}(\overline{\Omega})$ bounds for the sequence $\{v_{\varepsilon_n}\}_{n\geq 1} \subseteq W^{1,G}(\Omega)$. Hence, the Arzel\`{a}-Ascoli theorem guarantees that, up to a subsequence,
\begin{equation}\label{75568784}
v_{\varepsilon_n} \to v \quad \text{in} \quad C^{1}(\overline{\Omega})
\end{equation}
for some $v \in C^{1}(\overline{\Omega})$.\\
Recalling that $\Vert v_{\varepsilon_n}-u_0\Vert^{g^+}\leq \varepsilon_n$, for all $n\in\mathbb{N}$. So,
\begin{equation}\label{74568784}
v_{\varepsilon_n}\longrightarrow u_0\ \text{in}\ W^{1,G}(\Omega).
\end{equation}
Therefore, from $(\ref{75568784})$ and $(\ref{74568784})$, we obtain
 $v_{\varepsilon_n}\rightarrow u_0\ \ \text{in}\ C^{1}(\overline{\Omega}).$
So, for $n$ sufficiently large, say $n\geq n_0$, we have
$\Vert v_{\varepsilon_n}-u_0\Vert_{C^{1}(\overline{\Omega})}\leq r_0$
(where $r_0>0$ is defined in Definition \ref{def333}),
which provides
\begin{equation}\label{misss}
J(u_0)\leq J(v_{\varepsilon_n})\ \ \text{for all}\ n\geq n_0.
\end{equation}
On the other hand, we have
\begin{equation}\label{missss}
J(v_{\varepsilon_n})< J(u_0)\ \ \text{for all}\ n\in\mathbb{N}.
\end{equation}
Comparing $(\ref{misss})$ and $(\ref{missss})$, we reach a contradiction. This proves that $u_0$ is a local $W^{1,G}(\Omega)$-minimizer of $J$.
 This en{\rm d}s the proof.
\end{proof}
\begin{proof}[\textbf{Proof of Theorem \ref{thmC12}:}] The proof follows by applying  Propositions $\ref{prop3}$ and $\ref{prop4}$.
\end{proof}

\subsection*{Acknowledgments}
The research of V. D. R\u{a}dulescu is supported by the grant ``Nonlinear Differential Systems in Applied Sciences" of the Romanian Ministry of Research, Innovation and Digitization, within PNRR-III-C9-2022-I8/22.

\subsection*{Data availability statement}
Data sharing not applicable to this article as no data sets were generated or analysed during the current study.

\subsection*{Ethical Approval}
Not applicable.

\subsection*{Competing interests} The authors
read and approved the final manuscript. The authors have no relevant financial or non-financial interests to disclose.


\bigskip
\noindent \textsc{\textsc{anouar bahrouni}} \\
Mathematics Department, Faculty of Sciences, University of Monastir,
5019 Monastir, Tunisia\\
 (Anouar.Bahrouni@fsm.rnu.tn; bahrounianouar@yahoo.fr)

\bigskip
\noindent \textsc{\textsc{hlel missaoui}} \\
Mathematics Department, Faculty of Sciences, University of Monastir,
5019 Monastir, Tunisia\\
 (hlelmissaoui55@gmail.com)

 \bigskip
\noindent \textsc{\textsc{hichem ounaies}} \\
Mathematics Department, Faculty of Sciences, University of Monastir,
5019 Monastir, Tunisia\\
 (hichem.ounaies@fsm.rnu.tn)

\bigskip
\noindent \textsc{\textsc{vicen\c{t}iu d.\,r\u{a}dulescu}} \\
Faculty of Applied Mathematics, AGH University of Krak\'ow, 30-059 Krak\'ow, Poland \&
	Brno University of Technology, Faculty of Electrical Engineering and Communication, Technick\'a 3058/10, Brno
61600, Czech Republic \& Department of Mathematics,
University of Craiova, 200585 Craiova, Romania \& Simion Stoilow Institute of Mathematics of the Romanian Academy, Calea Grivi\c tei 21, 010702 Bucharest, Romania \\
(radulescu@inf.ucv.ro)
\end{document}